\documentclass{amsart}

\usepackage{amsmath, amssymb, amsfonts, amsbsy, amsthm, latexsym, stmaryrd, mathrsfs}
\usepackage{enumerate}
\usepackage{hyperref}
\usepackage{color}
\usepackage{hhline}
\usepackage{comment}
\usepackage{tikz}
\usepackage{ytableau}
\usepackage{caption}

\newcommand{\plabel}[2]{L_{#2}(#1)}

\newcommand{\parrays}[1]{\mathcal A_{#1}}
\newcommand{\wt}{\mathbf{wt}}
\newcommand{\ch}{\mathrm{char}}
\newcommand{\asc}{\mathrm{asc}}
\newcommand{\df}[1]{\mathrm{DF}_{#1}}
\newcommand{\dfa}[1]{\mathrm{DFA}_{#1}}
\newcommand{\arrayify}{\mathrm{array}}
\newcommand{\inc}{\operatorname{inc}}

\definecolor{magenta2}{HTML}{D81B60}
\definecolor{blue2}{HTML}{1E88E5}
\newlength\cellsize 

\savebox2{
	\begin{picture}(8,8)
		\put(0,0){\line(1,0){8}}
		\put(0,0){\line(0,1){8}}
		\put(8,0){\line(0,1){8}}
		\put(0,8){\line(1,0){8}}
\end{picture}}

\newcommand\boxify[1]{\def\thearg{#1}\def\nothing{}%
	\ifx\thearg\nothing\vrule width0pt height\cellsize depth0pt%
	\else\hbox to 0pt{\usebox2\hss}\fi%
	\vbox to \cellsize{\vss\hbox to \cellsize{\hss$_{#1}$\hss}\vss}}

\savebox3{
	\begin{picture}(8,8)
		\put(4,4){\circle{8}}
\end{picture}}

\newcommand{\circify}[1]{\def\thearg{#1}\def\nothing{}%
	\ifx\thearg\nothing\vrule width0pt height\cellsize depth0pt%
	\else\hbox to 0pt{\usebox3\hss}\fi%
	\vbox to \cellsize{\vss\hbox to \cellsize{\hss$_{#1}$\hss}\vss}}

\newcommand\nullify[1]{\def\thearg{#1}\def\nothing{}%
	\ifx\thearg\nothing\vrule width0pt height\cellsize depth0pt%
	\else\hbox to 0pt{\hss}\fi%
	\vbox to \cellsize{\vss\hbox to \cellsize{\hss$_{#1}$\hss}\vss}}

\newcommand\tableau[1]{\vtop{\let\\=\cr
		\setlength\baselineskip{-8000pt}
		\setlength\lineskiplimit{8000pt}
		\setlength\lineskip{0pt}
		\halign{&\boxify{##}\cr#1\crcr}}}

\newcommand\cirtab[1]{\vline\vtop{\let\\=\cr
		\setlength\baselineskip{-8000pt}
		\setlength\lineskiplimit{8000pt}
		\setlength\lineskip{0pt}
		\halign{&\circify{##}\cr#1\crcr}}}

\newcommand\nulltab[1]{\vtop{\let\\=\cr
		\setlength\baselineskip{-8000pt}
		\setlength\lineskiplimit{8000pt}
		\setlength\lineskip{0pt}
		\halign{&\nullify{##}\cr#1\crcr}}}

\def\N{{\mathbb{N}}}
\def\R{{\mathbb{R}}}
\def\Z{{\mathbb{Z}}}
\def\D{{\mathbb{D}}}

\newtheorem*{claim*}{Claim}
\newtheorem*{definition*}{Definition}
\newtheorem*{proposition*}{Proposition}
\newtheorem*{lemma*}{Lemma}
\newtheorem*{theorem*}{Theorem}
\newtheorem*{corollary*}{Corollary}

\newtheorem{theorem}{Theorem}[section]
\newtheorem{lemma}[theorem]{Lemma}
\newtheorem{corollary}[theorem]{Corollary}
\newtheorem{definition}[theorem]{Definition}
\newtheorem{remark}[theorem]{Remark}
\newtheorem{proposition}[theorem]{Proposition}

\title{A Crystal Analysis of $P$-Arrays}
\author{Henry Ehrhard}
\date{}

\begin{document}
	
\begin{abstract}
	Gasharov introduced the combinatorial objects known as $P$-arrays to prove $s$-positivity for the chromatic symmetric functions of incomparability graphs of (3+1)-free posets. We define a crystal, a directed colored graph with some additional axioms, on the set of $P$-arrays. The components of the crystal have $s$-positive characters, thereby refining the $s$-positivity theorems of Gasharov, as well as Shareshian and Wachs. The crystal hints at a possible generalization of the Robinson-Schensted correspondence applied to $P$-arrays.
\end{abstract}
\maketitle

\section{Introduction}
The chromatic symmetric function generalizes the chromatic polynomial of a graph \cite{Sta95}. The Stanley-Stembridge Conjecture \cite[Conjecture~5.5]{SS93} states that the chromatic symmetric function of the incomparability graph of a (3+1)-free poset is a positive sum of elementary symmetric functions. The conjecture has long motivated research into the chromatic symmetric functions of these graphs. 

In particular, Gasharov showed the weaker result that these symmetric functions are a positive sum of Schur functions, or are ``$s$-positive'' \cite{Gas96}. This was accomplished by reinterpreting the chromatic symmetric functions as generating functions for $P$-arrays, a combinatorial object corresponding to proper colorings. Moreover, the coefficients of the expansion into Schur functions can be stated in terms of $P$-arrays.

Stanley expressed the desire for a direct bijective proof of Gasharov's result, which would generalize the Robinson-Schensted correspondence \cite{Sta98}. A better understanding of Gasharov's result could help in addressing the natural suspicion that it holds for the larger class of claw-free graphs \cite[Conjecture~1.4]{Sta98}. Some thought has gone towards finding such a Robinson-Schensted correspondence for $P$-arrays \cite{Mag93},\cite{SWW97},\cite{Cho99},\cite{KP21} but the problem in its full generality remains open. Among the immediate consequences of this hypothetical correspondence would be a way to interpret each Schur function in the chromatic symmetric function as an explicit generating function for $P$-arrays in its own right. 

In this paper we similarly aim to refine the chromatic symmetric function as a sum of smaller $s$-positive generating functions. Crystals where introduced by Kashiwara \cite{Kas90} and have often served as a paradigm for interpreting and achieving $s$-positivity results, in part due to their connection with the representation theory of Lie groups. We thus construct a crystal on the set of $P$-arrays. The crystal determines a sign-reversing involution, similar to the one in Gasharov's proof, which we use to show that the crystal components have $s$-positive character. This refines the $s$-positivity theorem of Gasharov.

Guay-Paquet reduced the Stanley-Stembridge Conjecture to the case of unit interval orders \cite{Gua13}. More recently, Shareshian and Wachs' defined a quasisymmetric generalization of the chromatic symmetric function which, when applied to incomparability graphs of natural unit interval orders, suggests a strengthened version of the conjecture and yields an analogously refined $s$-positivity theorem \cite{SW16}. We show that our crystal is also applicable to this quasisymmetric setting, further refining the $s$-positivity theorem. Shareshian and Wachs use an involution similar to Gasharov's, but more efficient in some sense. By the same measure, our involution is optimal.

Our goals are similar to \cite{KP21} which, subject to additional poset avoidance conditions that are conjectured to be unnecessary, refines $s$-positivity of chromatic quasisymmetric functions for natural unit interval orders. The paper also presents a Robinson-Schensted correspondence in this case. In comparison, our work more generally deals with underlying (3+1)-free posets without additional constraints.

We begin in Section \ref{sec: chrom} by reviewing the necessary context around $P$-arrays and the chromatic symmetric function. In Section \ref{sec: crystals} we will explain what we mean by a crystal and discuss the diagram crystal defined in \cite{Ass-KC}. In Section \ref{sec: align} we will define the $r$ alignment of a $P$-array which is a sort of ``pairing rule'' that will allows us to define the crystal on $P$-arrays in Section \ref{sec: operators}. The $s$-positivity property of this crystal will be proved in Section \ref{sec: spos}, and its applicability to the chromatic quasisymmetric function for natural unit interval orders demonstrated in Section \ref{sec: nuio}. Finally, in Section \ref{sec: two-row} we will attempt to push our construction further by writing individual Schur functions in the expansion as generating functions when they correspond to partitions of length 1 or 2, even when our $P$-array crystal does not tell us how to do so. 

This last development relies heavily on the diagram crystal in \cite{Ass-KC}, and the proof is fundamentally bijective which means it can be viewed as a sort of partial Robinson-Schensted correspondence.  We are interested to see if the bijection can be extended to partitions of greater lengths. There is also the question of whether a crystal like ours can be defined on the proper colorings of an arbitrary claw-free graph, which would prove $s$-positivity of their chromatic symmetric functions.

\section{Chromatic Symmetric Functions}\label{sec: chrom}
A \textbf{proper coloring} of a graph $G=(V,E)$ is a map $\kappa:V\to \Z_+$ such that $\kappa(v)\ne\kappa(w)$ whenever $(v,w)\in E$. An \textbf{independent} subset of $V$ contains no two elements that form an edge. Proper colorings $\kappa:V\to\Z_+$ are then equivalently defined by requiring that $\kappa^{-1}(i)$ is an independent subset for all $i\in\Z_+$.
\begin{definition}\cite{Sta95}
	The \textbf{chromatic symmetric function} of a finite graph $G$ is given by \[X_G(x)=\sum_{\kappa} \prod_{v\in V} x_{\kappa(v)}\] where the sum is over all proper colorings of $G$.
\end{definition}
It is evident that these functions are indeed symmetric.

Given a poset $(P,\le_P)$, its \textbf{incomparability graph}, denoted $\inc(P)$, has vertex set $P$ and edges $(u,v)$ whenever $u$ and $v$ are incomparable in $P$, henceforth denoted by $u\parallel v$. See Fig. \ref{fig:poset} for an example of a poset and its incomparability graph. Notice that independent subsets of $\inc(P)$ are precisely the chains in $P$. That is, there is a canonical ordering on any independent subset. This observation motivates the following definition of objects that correspond to proper colorings of the incomparability graph.

\begin{figure}[ht]
	\begin{tikzpicture}
		\node (a) at (0,0) {$a$};
		\node (b) at (0,1) {$b$};
		\node (c) at (0,2) {$c$};
		\node (d) at (0,3) {$d$};
		\node (e) at (0,4) {$e$};
		\node (f) at (0,5) {$f$};
		\node (g) at (1,0) {$g$};
		\node (h) at (1,2) {$h$};
		\draw (a) to (b);
		\draw (b) to (c);
		\draw (c) to (d);
		\draw (d) to (e);
		\draw (e) to (f);
		\draw (e) to (f);
		\draw (g) to (b);
		\draw (h) to (d);
		\draw (h) to (a);
		\draw (h) to (g);
	\end{tikzpicture}
	\hspace{100pt}
	\begin{tikzpicture}
		\node (a) at (0,0) {$a$};
		\node (b) at (0,1) {$b$};
		\node (c) at (0,2) {$c$};
		\node (d) at (0,3) {$d$};
		\node (e) at (0,4) {$e$};
		\node (f) at (0,5) {$f$};
		\node (g) at (1,0) {$g$};
		\node (h) at (1,2) {$h$};
		\draw (c) to (h);
		\draw (b) to (h);
		\draw (a) to (g);
	\end{tikzpicture}

	\caption{\label{fig:poset} The Hasse diagram (left) and incomparability graph (right) of a (3+1)-free poset.}
\end{figure}
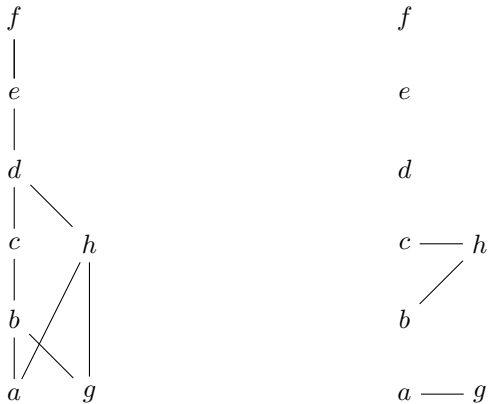

\begin{definition}\cite{Gas96}
	Let $(P,\le_P)$ be a finite poset. A \textbf{$P$-array} is an indexing $\{A_{i,j}\}$ of the elements of $P$ such that if $A_{i,j}$ is defined with $j>1$ then $A_{i,j-1}$ is also defined and $A_{i,j-1}<_P A_{i,j}$. We let $\parrays{P}$ denote the set of $P$-arrays.
\end{definition}
We visualize a $P$-array $\{A_{i,j}\}$ as the placement of the elements of $P$ on a grid indexed by matrix convention, where $A_{i,j}$ is placed in position $(i,j)$. Fig. \ref{fig:Parray} shows two $P$-arrays for the poset in Fig. \ref{fig:poset}.

\begin{figure}[ht]
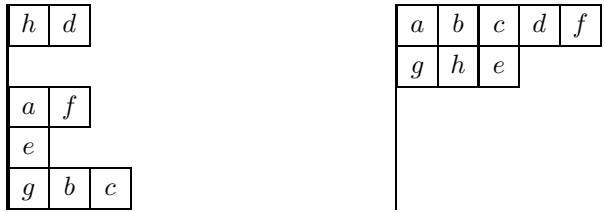

	\centering
	\begin{displaymath}
		\vline
		\ytableausetup{centertableaux}\begin{ytableau} 
			h&d \\ 
			\none\\
			a&f\\
			e\\
			g&b&c\\
		\end{ytableau} \hspace{100pt}
		\vline
		\ytableausetup{centertableaux}\begin{ytableau} 
			a&b&c&d&f\\ 
			g&h&e\\
			\none\\
			\none\\
			\none\\
		\end{ytableau}
	\end{displaymath}
	
	\caption{\label{fig:Parray} Two $P$-arrays with the one on the right being a $P$-tableau.}
\end{figure} 

Given a $P$-array $\{A_{i,j}\}$, the corresponding proper coloring of $\kappa:P\to \Z_+$ of $\inc(P)$ is defined by $\kappa(A_{i,j})=i$. Given a proper coloring $\kappa:P\to \Z_+$, the corresponding $P$-array is obtained by letting $A_{i,1},A_{i,2},\ldots,A_{i,k}$ be the elements of $\kappa^{-1}(i)$ in increasing order. The \textbf{weight} of a $P$-array $A$ is the weak composition $\wt(A)$ whose $i$th part is the number of elements in row $i$. If $\kappa$ is the corresponding proper coloring then $\wt(A)_i=\#\kappa^{-1}(i)$ so we realize that
\[
X_{\inc(P)}=\sum_{A\in\parrays P} \prod_{i\ge 1} x_i^{\wt(A)_i}.
\]

Developing these combinatorial objects further, we have the following definition.
\begin{definition}\cite{Gas96}
	Let $(P,\le_P)$ be a poset. A \textbf{$P$-tableau} is a $P$-array with the additional constraint that whenever $A_{i,j}$ is defined with $i>1$, then $A_{i-1,j}$ is also defined and $A_{i-1,j}\not>_P A_{i,j}$.
\end{definition}
The second $P$-array in Fig. \ref{fig:Parray} is a $P$-tableau. We see that the weight of a $P$-tableau is a partition in general, i.e. if $T$ is a $P$-tableau then $\wt(T)_i\ge \wt(T)_{i+1}$ for all $i\ge1$.

A poset is \textbf{($a$+$b$)-free}, for natural numbers $a$ and $b$, if it does not contain as an induced subposet the disjoint union of a $a$-chain with a $b$-chain. The poset in Fig. \ref{fig:poset} is (3+1)-free for instance.

Let $s_\lambda$ denote the Schur function indexed by a partition $\lambda$. We say a symmetric function is \textbf{$s$-positive} if it is a sum of Schur functions. A graph is said to be $s$-positive if its chromatic symmetric function is. The following theorem by Gasharov now reveals the motivation behind the definition of $P$-tableaux.
\begin{theorem}\cite{Gas96}\label{thm:gash}
	Let $(P,\le_P)$ be a (3+1)-free poset. Then $X_{\inc(P)}=\sum_\lambda c_\lambda s_\lambda$ where $c_\lambda$ is the number of $P$-tableaux of weight $\lambda$. In particular, $\inc(P)$ is $s$-positive.
\end{theorem}
The theorem and its proof don't give us an idea of which monomials ($P$-arrays) ``contribute'' to which Schur functions ($P$-tableaux) in any canonical way. Our primary goal in this paper is then to refine this result by developing a finer structure on $P$-arrays for a more localized explanation of $s$-positivity.

\section{Crystals}\label{sec: crystals}
A \textbf{crystal} is a colored directed graph consisting of the data $(\mathcal B, e_r, f_r, \wt)$. Here $\mathcal B$ is a vertex set. For all $r\ge 1$ the functions $e_r,f_r:\mathcal B\to\mathcal B\cup\{0\}$ called the \textbf {crystal raising and lowering operators} respectively, must satisfy $e_r(u)=v$ if and only if $f_r(v)=u$. The map $\wt:\mathcal B\to \Z^\ast$ is the \textbf{weight map} where $\Z^\ast$ is the set of $\Z$-valued sequences with finitely many nonzero entries. Letting $\alpha_r=(0,\ldots, 0,1,-1,0,\ldots)$ with the 1 as the $r$th entry, whenever $e_i(u)=v$ we also require that $\wt(v)=\wt(u)+\alpha_r$. The \textbf{character} of a crystal with vertex set $\mathcal B$ is
\[
\ch(\mathcal B)=\sum_{v\in\mathcal B} \prod_{i\ge 1} x_i^{\wt(v)_i}.
\]
Finally, a \textbf{highest weight element} of a crystal with vertex set $\mathcal B$ is some $v\in\mathcal B$ such that $e_r(v)=0$ for all $r\ge 1$. A more comprehensive description of crystals can be found in \cite{BS17}.

Be warned that we will overload the symbols $e_r,f_r$ and $\wt$ to be used on multiple vertex sets. However, there should be no ambiguity as it will be clear to which vertex set the maps are being applied.

In the following sections we will define a crystal $(\parrays P, e_r, f_r, \wt)$ for any (3+1)-free poset $(P,\le_P)$, with $\wt:\parrays{ P}\to \Z^\ast$ as already defined. We see that $\ch(\parrays P)$ coincides with $X_{\inc(P)}$ and is therefore $s$-positive by Theorem \ref{thm:gash}. Showing that the character of any connected component of the crystal is $s$-positive therefore generalizes this result, which is exactly what we intend to do. It will take us until Section \ref{sec: operators} to define the operators $e_r$ and $f_r$ for this crystal, and until Section \ref{sec: spos} to prove the refined $s$-positivity result.

\subsection{The Diagram Crystal}

For the remainder of this section we will discuss Assaf's diagram crystal \cite{Ass-KC}. This is not a prerequisite for understanding anything up through Section \ref{sec: nuio}. However, it may be a helpful example of a crystal with nice properties, and it is necessary for Section \ref{sec: two-row} where we partially succeed in pushing our refinement of Theorem \ref{thm:gash} even further.

By a \textbf{diagram} we mean a finite subset of $\Z_+\times\Z_+$ which we will interpret as positions in a grid using matrix convention. The elements of a diagram are referred to as its \textbf{cells}. We let $\mathcal D$ denote the set of diagrams. The weight map on $\mathcal D$ is defined so that $\wt(D)_i$ is the number of cells in row $i$ of $D$.

\begin{definition}\cite{Ass-KC}
	Let $D$ be a diagram and $r\in\Z_+$. The set of \textbf{$r$-pairs} of $D$ is a set of disjoint pairs of cells between rows $r$ and $r+1$ defined iteratively as follows. We say that two unpaired cells $x$ and $y$ in rows $r$ and $r+1$ respectively with $x$ weakly left of $y$ form an $r$-pair whenever every other cell in rows $r$ and $r+1$ in a column weakly between $x$ and $y$ is already part of an $r$-pair.
\end{definition}

\begin{definition}\cite{Ass-KC}
	We define $e_r:\mathcal D\to\mathcal D\cup\{0\}$ on $D\in\mathcal D$ as follows. If every cell in row $r+1$ of $D$ is $r$-paired then $e_r(D)=0$. Otherwise, take $(r+1,c)$ to be the rightmost cell in row $r+1$ of $D$ that is not $r$-paired and set $e_r(D)=(D\setminus \{(r+1,c)\} )\cup\{ (r,c) \}$. That is, we ``move'' the cell $(r+1,c)$ to row $r$.
\end{definition}
\begin{definition}\cite{Ass-KC}
	Given the above definition, the crystal lowering operator $f_r:\mathcal D\to\mathcal D\cup\{0\}$ is implicitly defined as follows. Let $D\in\mathcal D$. If every cell in row $r$ of $D$ is $r$-paired then $f_r(D)=0$. Otherwise, take $(r,c)$ to be the leftmost cell in row $r$ of $D$ that is not $r$-paired and set $f_r(D)=(D\setminus \{(r,c)\} )\cup\{ (r+1,c) \}$.
\end{definition}

Fig. \ref{fig:diag} shows the diagrams reachable from $\{ (1,1),(1,2),(2,2),(2,3) \}\in\mathcal D$ using only edges colored 1 and 2.

\begin{figure}[ht]
	\begin{tikzpicture}[scale=1.4]
		\node at (0,0) (A) {$\cirtab{ ~ & ~ & \\  & ~ & ~ \\ & & \\}$};
		\node at (1,-1) (B) {$\cirtab{ ~ & ~ & \\  &  & ~ \\ &~ & \\}$};
		\node at (2,-2) (C) {$\cirtab{ ~ & ~ & \\  &  &  \\ &~ &~ \\}$};
		\node at (1,-3) (D) {$\cirtab{  & ~ & \\ ~ &  &  \\ &~ &~ \\}$};
		\node at (0,-4) (E) {$\cirtab{  &  & \\ ~ & ~ &  \\ &~ &~ \\}$};
		\node at (0,-2) (F) {$\cirtab{  & ~ & \\ ~ &  & ~ \\ &~ & \\}$};
		\draw[thick,magenta2  ,->](A) -- (B)   node[midway,above]{$2$};
		\draw[thick,magenta2  ,->](B) -- (C)   node[midway,above]{$2$};
		\draw[thick,blue2  ,->](C) -- (D)   node[midway,above]{$1$};
		\draw[thick,blue2,->](D) -- (E)   node[midway,above ]{$1$};
		\draw[thick,magenta2,->](F) -- (D)   node[midway,above ]{$2$};
		\draw[thick,blue2,->](B) -- (F)   node[midway,above ]{$1$};
	\end{tikzpicture}
	\caption{\label{fig:diag} A portion of the diagram crystal.}
\end{figure}

We will also find useful the notion of column $c$-pairing which is defined similarly to $r$-pairing.
\begin{definition}\cite{Ass-KC}
	Let $D$ be a diagram and $c\in\Z_+$. The \textbf{column $c$-pairing} of $D$ is an iterative construction where we say that two cells $x$ and $y$ in columns $c$ and $c+1$ respectively with $x$ in a row weakly below that of $y$ are \textbf{column $c$-paired} whenever every other cell in columns $c$ and $c+1$ in a row weakly between $x$ and $y$ is already column $c$-paired.
\end{definition}

The following is a consequence of \cite[Theorem~4.2.5]{Ass-KC}.
\begin{proposition}\label{prop: col_pairs}
	The diagram crystal raising and lowering operators do not change the number of column $c$-pairs.
\end{proposition}
We will find this proposition applicable when a diagram $D$ shares a connected component in the crystal with a diagram $D'$ that is ``top-justified'', i.e. such that when $(r,c)\in D$ with $r>1$ we have $(r-1,c)\in D$ as well. In this case, Proposition \ref{prop: col_pairs} implies every cell in the shorter of the two columns $c$ and $c+1$ is column $c$-paired.

A second invariant of components in the diagram crystal that we will find useful is given by the next lemma.
\begin{lemma}\label{lem: max_descent}
	Let $D$ be a diagram and consider sequences of cells $(r_1,c_1),\ldots,(r_k,c_k)$ in $D$ such that $r_i<r_{i+1}$ and $c_i\le c_{i+1}$. The maximal length of such a sequence is unchanged by the crystal raising and lowering operators.
\end{lemma}
\begin{proof}
	Suppose $(r_1,c_1),\ldots,(r_k,c_k)\in D$ is a sequence of cells in a diagram $D$ with $r_i<r_{i+1}$ and $c_i\le c_{i+1}$ of maximal length. We will show that there is such a sequence of cells of length $k$ in $e_r(D)$ if $e_r(D)\ne0$. 
	
	The only way this might not be the case is if $r+1=r_i$ for some $1\le i\le k$ and $e_{r}$ were to remove $(r_i,c_i)$ from $D$. In this case, we have $(r_i-1,c_i)\in e_r(D)$. Unless $i>1$ and $r_{i-1}=r_i-1$, we get another sequence of cells in $e_r(D)$ with the desired properties by simply replacing $(r_i,c_i)$ in the sequence with $(r_i-1,c_i)$. So assume $i>1$ and $r_{i-1}=r_i-1$.
	
	We know $(r_i,c_i)$ is not $r$-paired in $D$. Then $(r_{i-1},c_{i-1})$ must be $r$-paired with some $(r_i, d)$ with $c_{i-1}\le d< c_i$. Replacing $(r_i,c_i)$ in the sequence $(r_1,c_1),\ldots,(r_k,c_k)$ with $(r_i,d)$ then gives us a sequence in $e_r(D)$ of length $k$ with the desired properties.
	
	It is similar to show that if $f_r(D)\ne 0$ then $f_r(D)$ has such a sequence of cells of length $k$.
\end{proof}
Once again considering the cases of a component containing a top-justified diagram, the lemma says that the maximal length of a sequence of cells as described above is the maximal number of cells in a column.

Assaf's Theorem 5.3.4 specializes to the following statement.

\begin{theorem}\cite{Ass-KC}\label{thm: diagrams}
	Let $D$ be a top-justified diagram. Then the character of the component of the diagram crystal containing $D$ is $s_{\wt(D)}$. Additionally, $D$ is the unique highest weight element in its connected component.
\end{theorem}

\begin{remark}\label{rem: young}
	The theorem as presented in \cite{Ass-KC} actually shows that there is a crystal isomorphism between the relevant components of the diagram crystal and a tableaux crystal. The diagrams in a given component are then in weight preserving bijection with semi-standard Young tableaux of shape $\wt(D)$.
\end{remark}

\section{The $r$ Alignment}\label{sec: align}
For the remainder of this paper, $(P,\le_P)$ is always a finite (3+1)-free poset. In this section we will introduce the $r$ alignment of a $P$-array. The $r$ alignment will be fundamental for defining the crystal operators on $P$-arrays in Section \ref{sec: operators}. It is a way of horizontally spacing out the elements of rows $r$ and $r+1$, and can be thought of as a ``pairing rule,'' analogous to the idea of $r$-pairs for the diagram crystal, which helps us determine how $e_r$ and $f_r$ affect rows $r$ and $r+1$.

\begin{definition}
	Let $A$ be a $P$-array and let $r\ge 1$. Let the chains $a_1<_P\cdots<_Pa_m$ and $b_1<_P\cdots<_P b_n$ be the elements of rows $r$ and $r+1$ of $A$ respectively. Let $C:\{a_1,\ldots,a_m,b_1,\ldots,b_n \}\to\Z_{+}$ be the function inductively defined so that $C(b_k)=k$ for each $1\le k\le n$, and
	\[
	C(a_k)=\max(\{C(b_i)\mid b_i<_P a_k  \}\cup \{ C(a_i) \mid 1\le i<k \})+1.
	\]
	
	Then we define the \textbf{$r$ pre-alignment} of $A$ to be the map \[\{a_1,\ldots,a_m,b_1,\ldots,b_n \}\to\{r,r+1\}\times\Z_{+}\] such that each $a_k$ maps to $(r,C(a_k))$, and each $b_k$ to $(r+1, C(b_k))$.
\end{definition}

Fig. \ref{fig:prealign} shows a visualization of the $r$ pre-alignment for a $P$-array $A$ if $P$ is as in Fig. \ref{fig:poset}, row $r$ of $A$ contains the elements $b,c,d$ and row $r+1$ contains $a,h$.

\begin{figure}[ht]
\ytableausetup{mathmode, centertableaux,boxsize=1.5em}\begin{ytableau} 
	\none[r]&\none&\none&b&c&d\\\none[r+1]&\none&a&h
\end{ytableau} 
	
	\caption{\label{fig:prealign} An $r$ pre-alignment for a $P$-array with $P$ as in Fig. \ref{fig:poset}.}
\end{figure}

\begin{definition}\label{defn: align}
		Let $A$ be a $P$-array and let $r\ge 1$. Let the chains $a_1<_P\cdots<_Pa_m$ and $b_1<_P\cdots<_P b_n$ be the elements of rows $r$ and $r+1$ of $A$ respectively. Let $\phi_0$ be the $r$ pre-alignment. We construct the \textbf{$r$ alignment} of $A$ as follows.
		
		Suppose we have some $\phi_k: \{a_1,\ldots,a_m,b_1,\ldots,b_n \}\to\{r,r+1\}\times\Z_{+}$. Select the rightmost element $x$ mapped to some $(i,c)$ such that column $c+1$ of $\phi_k$ is nonempty and contains no $y>_Px$, if such an $x$ exists. Then we define \[\phi_{k+1}:\{a_1,\ldots,a_m,b_1,\ldots,b_n \}\to\{r,r+1\}\times\Z_{+}\] so that $\phi_{k+1}(x)=(i,c+1)$ and $\phi_{k+1}$ coincides with $\phi_k$ elsewhere. If no such $x$ exists then the $r$ alignment of $A$ is defined to be $\phi_k$.
\end{definition}

Fig. \ref{fig:align} shows the $r$ alignment for a $P$-array with rows $r$ and $r+1$ as in Fig. \ref{fig:prealign}.

\begin{figure}[ht]
	\ytableausetup{mathmode, centertableaux,boxsize=1.5em}\begin{ytableau} 
		\none[r]&\none&\none&b&c&d\\\none[r+1]&\none&a&\none&h
	\end{ytableau} 
	
	\caption{\label{fig:align} The $r$ alignment for a $P$-array with $P$ as in Fig. \ref{fig:poset}.}
\end{figure}

\begin{proposition}
	The $r$ alignment of a $P$-array $A$ is well-defined. Additionally, if $a_1<_P\cdots<_Pa_m$ and $b_1<_P\cdots<_P b_n$ are the elements of rows $r$ and $r+1$ of $A$ respectively, then in the $r$ alignment each $a_i$ is in a column strictly left of $a_{i+1}$ and each $b_i$ is in a column strictly left of $b_{i+1}$.
\end{proposition}
\begin{proof}
	Our attention is restricted entirely to rows $r$ and $r+1$ of $A$. That the $r$ pre-alignment $\phi_0$ satisfies the condition on $a_1,\ldots, a_m$ and $b_1,\ldots, b_n$ is immediate from the definition. Let $d$ be the rightmost occupied column in the $r$ pre-alignment. Suppose we have $\phi_k$ as in Definition \ref{defn: align} that satisfies the condition on $a_1,\ldots, a_m$ and $b_1,\ldots, b_n$, and the rightmost occupied column of $\phi_k$ is $d$.
	
	Should these inductive hypotheses continue to hold, we will have shown that the process to determine the $r$ alignment must terminate, as we are limited by $d$ in how far we can shift each element to the right. The other obstacle to well-definiteness is the question of whether $x$, as in the definition, is unique when it exists. To this end, consider some column $c$ in $\phi_k$ which contains two elements, some $a_i$ and $b_j$. If column $c+1$ contains an element in $\phi_k$ then by assumption it is either $a_{i+1}$ or $b_{j+1}$. Thus, there is at most one element in column $c$ for which there is no greater element in column $c+1$. This indeed proves that $x$ is uniquely determined when it exists.
	
	Now suppose $x$ exists and resides in some column $c$ of $\phi_k$. By choice of $x$, column $c+1$ is nonempty in $\phi_k$, so $c<d$. In $\phi_{k+1}$ we know $x$ is mapped to column $c+1$, so any column right of $d$ remains unoccupied. Moreover, if $x=a_i$ for some $1\le i\le m$, then by the inductive hypothesis $a_{i-1}$ (if it exists) is strictly left of column $c$ in both $\phi_k$ and $\phi_{k+1}$. If $a_{i+1}>_P x$ exists then it is strictly right of column $c+1$ of $\phi_k$ and $\phi_{k+1}$ using both the inductive hypothesis and the choice of $x$. The same argument can be made when $x=b_i$ for some $1\le i\le n$. Then the inductive hypothesis on $a_1,\ldots, a_m$ and $b_1,\ldots, b_n$ continues to hold.
\end{proof}

\begin{lemma}\label{lem:sw}
	In the $r$ alignment of a $P$-array $A$, any element in row $r+1$ that does not share a column with an element in row $r$ is strictly left of any element in row $r$ that does not share a column with an element in row $r+1$.
\end{lemma}
\begin{proof}
	This is clear for the row $r$ pre-alignment. Prior to moving an element from some column $c$ to column $c+1$ while obtaining the $r$ alignment, column $c$ must contain an element in each row while column $c+1$ contains exactly one element. Thus, the property is maintained as we move elements to the right.
\end{proof}

In Section \ref{sec: operators} we will define our crystal operators on $\parrays{P}$ using the $r$ alignment. To prove that our operators are essentially inverses as required, we will want to see that they act predictably on the $r$ alignment. This motivates us to write a more static characterization of the $r$ alignment, as opposed to the procedural definition.

\begin{definition}\label{defn: weak_align}
	Let $A$ be a $P$-array. Let the chains $a_1<_P a_2<_P\cdots<_Pa_m$ and $b_1<_Pb_2<_P\cdots<_P b_n$ be the elements of rows $r$ and $r+1$ of $A$ respectively. A \textbf{weak $r$ alignment} of $A$ is a map $\phi:\{a_1,\ldots,a_m,b_1,\ldots,b_n \}\to\{r,r+1\}\times\Z_{+}$ satisfying the following properties:
	\begin{enumerate}
		\item each $a_i$ is mapped to row $r$ in a column strictly right of $a_{i-1}$ if it exists, and each $b_i$ is mapped to row $r+1$ in a column strictly right of $b_{i-1}$ if it exists,
		\item if $\phi$ maps an element to some column $c>1$, then it also maps an element to column $c-1$,
		\item if $x<_P y$ with $x$ mapped to row $r+1$ and $y$ mapped to row $r$ then $x$ is mapped to a column strictly left of $y$, and
		\item if some $x$ is mapped to some column $c$, then either $\phi$ maps some $y>_Px$ to column $c+1$ or it maps no elements to column $c+1$.
	\end{enumerate}
\end{definition}

\begin{definition}
	Let $A$ be a $P$-array. Let the chains $a_1<_P a_2<_P\cdots<_Pa_m$ and $b_1<_Pb_2<_P\cdots<_P b_n$ be the elements of rows $r$ and $r+1$ of $A$ respectively. Let $\phi,\psi:\{a_1,\ldots,a_m,b_1,\ldots,b_n \}\to\{r,r+1\}\times\Z_{+}$. We write $\phi\preceq^* \psi$ when the position of each element in $\phi$ is in a column weakly left of its position in $\psi$. This defines a partial order on such maps. If $\phi$ and $\psi$ are weak $r$ alignments of $A$ we write $\phi\preceq \psi$.
\end{definition}

\begin{proposition}\label{prop:algn_char}
	Let $A$ be a $P$-array and $r\ge 1$. The $r$ alignment is the unique minimal weak $r$ alignment according to the partial order $\preceq$.
\end{proposition}
\begin{proof}
	That the $r$ alignment satisfies (1) and (4) follows by definition. Property (3) is clear in the pre-alignment, and $x$ can never be shifted right into a column containing $y$. So (3) holds in the alignment as well.
	
	Let $\phi_0$ denote the $r$ pre-alignment. For every element in a column $c>1$ of $\phi_0$ there is a lesser element in column $c-1$. Therefore there is some chain of elements in consecutive columns $x_1<_P\cdots<_P x_\ell$ where $x_1$ is in the first column and $x_\ell$ is in the rightmost nonempty column, call it $d$. No element can move right of $d$ as we construct the $r$ alignment, and therefore each element in the chain $x_1<_P\cdots<_P x_\ell$ is in the same position in the $r$ alignment. This proves that the $r$ alignment satisfies (2), and is therefore a weak $r$ alignment.
	
	Now let $\psi$ be a weak $r$ alignment. For any $a_k$ there is some sequence \[b_1,\ldots, b_j,a_i,\ldots, a_k\] (possibly with $j=0$) that is increasing in $\le_P$ and occupy consecutive columns in the $r$ pre-alignment, similarly to above. By properties (1) and (3), each element in this sequence must be in a column strictly right of the preceding one in $\psi$. Thus, the position of $a_k$ in $\psi$ is weakly right of its position in the $r$ pre-alignment. The same is certainly true for each $b_k$ so $\phi_0\preceq^\ast \psi$.
	
	Let $\phi_k$ be an intermediate step between the $r$ pre-alignment and $r$ alignment as in Definition \ref{defn: align}, and assume $\phi_k\preceq^\ast \psi$. Suppose there exists a rightmost $x$ in $\phi_k$ such that the adjacent column to its right is nonempty and contains no element greater than $x$. We will show $x$ lies strictly right in $\psi$ of its position in $\phi_k$. This is true if $x$ is in the rightmost nonempty column of $\psi$, which is weakly right of the rightmost nonempty column of $\phi_k$, which is in turn strictly right of $x$ in $\phi_k$. Otherwise, there is some $y>_Px$ one column right of $x$ in $\psi$ by (4).
	
	Say $x$ lies in column $c$ in $\phi_k$. If $y$ lies in a column strictly right of $c$ in $\phi_k$, it must be in column $c+i$ for some $i\ge 2$ by choice of $x$. Then $y$ lies strictly right of column $c+i-1$ in $\psi$ which means $x$ lies strictly right of $c+i-2\ge c$ in $\psi$.
	
	Suppose instead $y$ lies in some column weakly left of $c$ in $\phi_k$. Then $x$ and $y$ lie in opposite rows. However, this means any element right of $x$ in $\phi_k$ is greater than $x$ which contradicts our choice of $x$. By induction we can now say that $\phi_{k+1}\preceq^\ast \psi$. Then if $\phi_k$ is the $r$ alignment we have $\phi_k\preceq \psi$.
\end{proof}

Using Proposition \ref{prop:algn_char}, we end the section with a final useful lemma.
\begin{lemma}\label{lem:col_ordering}
	If the $r$ alignment of $A$ maps a unique element $x\in P$ to some column $c$, then any element strictly left of $c$ is smaller than any element weakly right of $c$ in the $r$ alignment.
\end{lemma}
\begin{proof}
    Let $y$ be an element strictly left of $c$, and $z$ an element weakly right of $c$. If $y$ and $z$ share a row, then $y<_P z$ by (1) in Definition \ref{defn: weak_align}. We have $y<_Px$ by (4) in Definition \ref{defn: weak_align}, so if $z$ shares a row with $x$ we have $y<_P z$ as well. Thus, assume $y$ shares a row with $x$ and $z$ does not.
    
    If $z$ is in row $r+1$, let $z'$ be the element in the same column of row $r$ which exists by Lemma \ref{lem:sw}. We have $z\not<_P z'$ by (3) in Definition \ref{defn: weak_align} so $y<_P z$ by the (3+1)-free condition.
    
    Otherwise, it suffices to assume $z$ is the leftmost element in row $r$ in a column strictly right of $c$. If $z$ does not share a column with any element in row $r$ then once again we have $y<_P z$ by (4) in Definition \ref{defn: weak_align}. So suppose we have some $z'$ in the same column as $z$ in row $r+1$. If we had $y\not <_P z$ in this circumstance, that would imply $z<_P z'$ by the (3+1)-free condition, and also $x\not<_P z$. Then moving the position of $z$ one column to the left would still satisfy conditions (1)-(4) in Definition \ref{defn: weak_align}, contradicting Proposition \ref{prop:algn_char}.
\end{proof}

\section{The $P$-Array Crystal}\label{sec: operators}
We now have the tools to define the crystal operators on $\parrays{P}$.
\begin{definition}
	The \textbf{crystal lowering operator} $f_r:\parrays{P}\to\parrays{P}\cup\{0\}$ acts on $A\in\parrays{P}$ as follows. Let $a_1<_P\cdots<_Pa_m$ and $b_1<_P\cdots<_P b_n$ be the entries of rows $r$ and $r+1$ of $A$ respectively.
	\begin{itemize}
		\item If every column of the $r$ alignment with an entry in row $r$ also contains an entry in row $r+1$ then define $f_r(A)=0$.
		\item Otherwise let $p$ be minimal such that $a_p$ does not share a column with an element in row $r+1$ in the $r$ alignment. Let $t\ge 0$ be minimal such that there is no $b_i>_Pa_{p+t}$ one column right of $a_{p+t}$. Then we move $a_p,\ldots, a_{p+t}$ to row $r+1$, and any $b_i$ that shares a column with one of these entries to row $r$.
	\end{itemize}
\end{definition}

\begin{proposition}\label{prop: lower_def}
	The crystal lowering operator $f_r$ is well-defined, and if we have $A\in\parrays{P}$ with $f_r(A)\ne 0$ then $\wt(f_r(A))=\wt(A)-\alpha_r$.
\end{proposition}
\begin{proof}
	Let $A\in\parrays{P}$ and let $a_1<_P\cdots<_Pa_m$ and $b_1<_P\cdots<_P b_n$ be the entries of rows $r$ and $r+1$ of $A$ respectively. Select $a_p$ as in the definition, assuming it exists. By Lemma \ref{lem:sw} every nonempty column strictly right of the column containing $a_p$ in the $r$ alignment must contain an entry in row $r$. In particular, there is no $b_i$ in any column strictly right of $a_m$ and it therefore makes sense to select $t\ge 0$ as in the definition. The entries $a_p,\ldots,a_{p+t}$ reside in consecutive columns of the $r$ alignment by Lemma \ref{lem:sw} and property (1) of weak $r$ alignments, and for each $0\le i< t$ there is some $b_j\not>_P a_{p+i}$ that shares a column with $a_{p+i+1}$ by choice of $t$. Let $q\ge 1$ be minimal such that either $b_q$ lies in a column strictly right of $a_p$ or does not exist. The relevant columns of the $r$ alignment are then as in Fig. \ref{fig:lower} (which degenerates to just the column containing $a_p$ if $t=0$).
	
	\begin{figure}[ht]
		\ytableausetup{mathmode, centertableaux,boxsize=2.5em}\begin{ytableau} 
			\scriptstyle a_p&\scriptstyle a_{p+1}&\scriptstyle a_{p+2}&\cdots&\scriptstyle a_{p+t}\\ 
			\none&\scriptstyle b_q&\scriptstyle b_{q+1}&\cdots&\scriptstyle b_{q+t-1}
		\end{ytableau} 
			
		\caption{\label{fig:lower}The columns in the alignment of $A$ whose entries are to be row swapped by $f_r$.}
	\end{figure}
	
	We must show
	\begin{enumerate}[(i)]
		\item $b_{q-1}<_P a_p$
		\item $a_{p+t}<_P b_{q+t}$ 
		\item $a_{p-1}<_P b_q$
		\item $b_{q+t-1}<_P a_{p+t+1}$
	\end{enumerate}
	whenever these indices are valid. The point being that (i) and (ii) show row $r+1$ of $f_r(A)$ is a chain, and (iii) and (iv) show row $r$ of $f_r(A)$ is a chain. When $t=0$ we need not consider (iii) or (iv).
	
	We have (i) immediately from Lemma \ref{lem:col_ordering}. If $b_{q+t}$ is one column right of $a_{p+t}$ then (ii) follows by choice of $a_{p+t}$. Otherwise, by Lemma \ref{lem:sw} there is some $a_{p+t+i}$ with $i>2$ that shares a column with $b_{q+t}$.  By (3) in Definition \ref{defn: weak_align}, $b_{q+t}\not<_P a_{p+t+i}$ so we get (ii) by the (3+1)-free condition.
	
	If $t\ne 0$ we know $a_{p+t-1}\not<_P b_{q+t-1}$ so (iv) follows from the (3+1)-free condition. By property (3) in Definition \ref{defn: weak_align} we have $b_q\not<_P a_{p+1}$ so (iii) follows from the (3+1)-free condition.
	
	Fig. \ref{fig:lower} makes clear that $\wt(f_r(A))=\wt(A)-\alpha_r$.
\end{proof}

\begin{lemma}\label{lem:flip_lower}
	Let $A$ be a $P$-array. Let $a_p<_P\cdots<_Pa_{p+t}$ and $b_q<_P\cdots<_Pb_{q+t-1}$ be the elements whose rows are swapped by $f_r$ as in Fig. \ref{fig:lower}. Then $a_{p+i+1}\parallel b_{q+i}$ for each $0\le i<t$.
\end{lemma}
\begin{proof}
	We have $a_{p+i}\not<_P b_{q+i}$ by choice of $a_{p+t}$, thus $a_{p+i+1}\not<_Pb_{q+i}$. Since $a_{p+i+1}$ shares a column with $b_{q+i}$ in the $r$ alignment, we have $a_{p+i+1}\not>_Pb_{q+i}$ by property (3) in Definition \ref{defn: weak_align}.
\end{proof}

As with the crystal lowering operator, we will soon define the crystal raising operator using the $r$ alignment of a $P$-array. To eventually see that these two operations are inverses, we then want to know how $f_r$ affects the $r$ alignment. The answer is given by the following lemma.
\begin{lemma}\label{lem: fcol}
	Applying $f_r$ to a $P$-array $A$ does not change the column of any entry in the $r$ alignment.
\end{lemma}
\begin{proof}
	
	Let $a_1<_P a_2<_P\cdots<_Pa_m$ and $b_1<_Pb_2<_P\cdots<_P b_n$ be the entries of rows $r$ and $r+1$ of $A$ respectively. Proposition \ref{prop: lower_def} demonstrates that the entries of rows $r$ and $r+1$ in $f_r(A)$ are respectively given by\[a_1,\ldots,a_{p-1}, b_q,\ldots, b_{q+t-1}, a_{p+t+1},\ldots, a_m\] and \[ b_1,\ldots, b_{q-1},a_p,\ldots, a_{p+t},b_{q+t},\ldots, b_n\] for $p$ and $t$ as in the definition of $f_r$, and $q\ge 1$ minimal such that $b_q$ either does not exist or lies in a column strictly right of $a_p$ in the $r$ alignment of $A$. Define
	\[
	\phi:\{a_1,\ldots,a_m,b_1,\ldots,b_n \}\to\{r,r+1\}\times\Z_{+}
	\]
	so that it agrees with the $r$ alignment, except on $a_p,\ldots a_{p+t}$ which are sent to row $r+1$, and $b_q,\ldots,b_{q+t-1}$ which are sent to row $r$, without changing column assignments. This amounts to swapping the rows in Fig. \ref{fig:lower}. We must show that $\phi$ is the $r$ alignment of $f_r(A)$.
	
	We see that $\phi$ inherits from the $r$ alignment of $A$ properties (1), (2), and (4) in Definition \ref{defn: weak_align}. The only entries in row $r$ of $\phi$ weakly left of some $b_i$ with $1\le i< q$ are among $a_1,\ldots, a_{p-1}$. All these entries occupy the same position in the $r$ alignment of $A$ so cannot violate (3). For $q+t\le i\le n$, the only entries in row $r$ of $\phi$ left of $b_i$ that do not also occupy row $r$ in the $r$ alignment of $A$ are $b_j$ for $q\le j< q+t$ which are less than $b_i$ and therefore do not violate (3) in Definition \ref{defn: weak_align}. The only elements of row $r$ left of $a_p$ in $\phi$ are $a_i$ that are lesser than $a_p$. Each $a_{p+i}$ for $1\le i\le t$ shares a column with an incomparable element by Lemma \ref{lem:flip_lower}, so cannot violate (3). So $\phi$ is a weak $r$ alignment of $f_r(A)$.
	
	Suppose we have a weak $r$ alignment $\psi$ of $f_r(A)$ such that $\psi\preceq \phi$. We will show $\psi=\phi$. Since $a_p$ is the leftmost element in row $r$ of the $r$ alignment of $A$ that does not share a column with an entry in row $r+1$, we know from Lemma \ref{lem:sw} that $b_1,\ldots, b_{q-1},a_p,\ldots,a_m$ occupy adjacent columns in the $r$ alignment of $A$, hence in $\phi$. Properties (1) and (3) of weak $r$ alignments force each element in this chain to be strictly right of the previous element in $\psi$, so the elements of this chain occupy the same columns in $\psi$ and $\phi$.
	
	Take $b_i$ for $q+t\le i\le n$ and suppose $b_j$ occupies the same column in $\psi$ and $\phi$ for any $j>i$. The only elements greater than $b_i$ are those $b_j$, and possibly some $a_k$ for $k> p+t$. We have assumed that each such greater element occupies the same column in $\psi$ and $\phi$. Recalling that $a_p$ is the leftmost element in row $r$ of the $r$ alignment of $A$ that does not share a column with an entry in row $r+1$, Lemma \ref{lem:sw} precludes $b_i$ from being strictly right of $a_m$ in the $r$ alignment, hence in $\phi$. Then in $\phi$, (3) and (4) place $b_i$ one column left of the leftmost entry greater than $b_i$. Property (4) forces $b_i$ to occupy the same position in $\psi$.
	
	Since by choice of $b_q$ we have $a_{p+i}\not<_P b_{q+i}$ for each $0\le i<t$, we know $b_{q+i}<_P a_{p+i+2}$ by the (3+1)-free condition (if $a_{p+i+2}$ exists). We also have $b_{q+i}\parallel a_{p+i+1}$ by Lemma \ref{lem:flip_lower}. This means that in $\phi$, each such $b_{q+i}$ is one column left of the leftmost entry greater than $b_{q+i}$ if such an entry exists. We have already determined that $\phi$ and $\psi$ agree on the positions of entries greater than any of $b_q,\ldots, b_{q+t-1}$, other than these entries themselves. Now properties (1) and (4) of weak $r$ alignments prohibit $\psi$ from placing any of these $b_{q+i}$ strictly left of its position in $\phi$.
	
	It remains only to show that $\phi$ and $\psi$ agree on $a_1,\ldots, a_{p-1}$. Define \[\xi:\{ a_1,\ldots, a_m,b_1,\ldots,b_n \}\to\{r,r+1\}\times\Z_{+}\] to agree with the $r$ alignment of $A$ (hence the column assignments of $\phi$), except on the positions of $a_1,\ldots, a_{p-1}$ where it instead agrees with $\psi$. We know $\xi$ places each $a_i$ left of $a_{i+1}$ so it satisfies (1). Column assignments in $\psi$ all agree with $\xi$ which therefore satisfies (2) and (4).
	
	If $b_j<_P a_i$ for some $p\le i\le m$ and $1\le j\le n$ then the positions of these entries in $\xi$ agrees with the $r$ alignment of $A$ and therefore they do not violate (3). For $1\le i<p$ we can only have $b_j<_P a_i$ for some $j<q$. The positions of these elements in $\xi$ agree with $\psi$ so cannot violate (3) either. Then $\xi$ is a weak $r$ alignment that precedes the $r$ alignment of $A$ according $\preceq$. Then $\xi$ is, in fact, the $r$ alignment of $A$ by Proposition \ref{prop:algn_char}. In particular, for $i<p$ we have $\psi(a_i)=\xi(a_i)=\phi(a_i)$. Therefore $\phi=\psi$ so $\phi$ is the alignment of $f_r(A)$ by Proposition \ref{prop:algn_char}.
\end{proof}

We next give a series of analogous statements for the crystal raising operator. The proofs are mostly, but not entirely, symmetric.
\begin{definition}
	The \textbf{crystal raising operator} $e_r:\parrays{P}\to\parrays{P}\cup\{0\}$ acts on $A\in\parrays{P}$ as follows. Let $a_1<_P\cdots<_Pa_m$ and $b_1<_P\cdots<_P b_n$ be the entries of rows $r$ and $r+1$ of $A$ respectively.
	\begin{itemize}
		\item If every column of the $r$ alignment with an entry in row $r+1$ also contains an entry in row $r$ then define $e_r(A)=0$.
		\item Otherwise let $p$ be maximal such that $b_p$ does not share a column with an entry in row $r$ in the $r$ alignment. Let $t\ge 0$ be minimal such that there is no $a_i>_Pb_{p+t}$ one column right of $b_{p+t}$. Then we move $b_p,\ldots, b_{p+t}$ to row $r$, and any $a_i$ that shares a column with one of these entries to row $r+1$.
	\end{itemize}
\end{definition}

\begin{proposition}\label{prop: raise_def}
	The crystal raising operator $e_r$ is well-defined, and if we have $A\in\parrays{P}$ with $e_r(A)\ne 0$ then $\wt(e_r(A))=\wt(A)+\alpha_r$.
\end{proposition}
\begin{proof}
	Let $A\in\parrays{P}$ and let $a_1<_P\cdots<_Pa_m$ and $b_1<_P\cdots<_P b_n$ be the entries of rows $r$ and $r+1$ of $A$ respectively. Select $b_p$ as in the definition, assuming it exists. If the column immediately right of $b_n$ in the $r$ alignment contains some $a_i$, then $b_n<_P a_i$ by Lemma \ref{lem:col_ordering}, and it therefore makes sense to select $t\ge 0$ as in the definition of $e_r$. The entries $b_p,\ldots,b_{p+t}$ reside in consecutive columns of the $r$ alignment by choice of $b_p$ and property (3) of weak $r$ alignments, and for each $0\le i<t$ there is some $a_j\not>_P b_{p+i}$ that shares a column with $b_{p+i+1}$. Let $q\ge 1$ be minimal such that either $a_q$ lies in a column strictly right of $b_p$ or does not exist. The relevant columns of the $r$ alignment are then as shown in Fig. \ref{fig:raise} (which degenerates to just the column containing $b_p$ if $t=0$).
	
	\begin{figure}[ht]
		\ytableausetup{mathmode, centertableaux,boxsize=2.5em}\begin{ytableau} 
			\none&\scriptstyle a_q&\scriptstyle a_{q+1}&\cdots&\scriptstyle a_{q+t-1}\\
			\scriptstyle b_p&\scriptstyle b_{p+1}&\scriptstyle b_{p+2}&\cdots&\scriptstyle b_{p+t}\
		\end{ytableau} 
			
			\caption{\label{fig:raise}The columns in the alignment of $A$ whose elements are to be row swapped by $e_r$.}
		\end{figure}
	
	To show that both rows of $e_r(A)$ are chains, it suffices to show
	\begin{enumerate}[(i)]
		\item $a_{q-1}<_P b_p$
		\item $b_{p+t}<_P a_{q+t}$ 
		\item $b_{p-1}<_P a_q$
		\item $a_{q+t-1}<_P b_{p+t+1}$
	\end{enumerate}
	whenever these indices are valid. When $t=0$ we need not consider (iii) or (iv).
	
	We have (i) immediately from Lemma \ref{lem:col_ordering}. By choice of $b_p$ we know that $a_{q+t}$, if it exists, lies one column right $b_{p+t}$. Then (ii) follows by choice of $b_{p+t}$.
	
	If $t\ne 0$ we have $b_{p+t-1}\not<_P a_{q+t-1}$ so (iv) follows from the (3+1)-free condition. We also know $b_p\not<_P a_q$ so we must have $a_q\not<_Pb_{p+1}$ or we could violate the characterization of the $r$ alignment of $A$ given in Proposition \ref{prop:algn_char} by shifting $a_q$ one position to the left. Thus (iii) follows from the (3+1)-free condition.
	
	Fig. \ref{fig:raise} makes clear that $\wt(e_r(A))=\wt(A)+\alpha_r$.
\end{proof}

\begin{lemma}\label{lem:flip_raise}
	Let $A$ be a $P$-array. Let $b_p<_P\cdots<_Pb_{p+t}$ and $a_q<_P\cdots<_Pa_{q+t-1}$ be the elements whose rows are swapped by $e_r$ as in Fig. \ref{fig:raise}. Then $b_{p+i+1}\parallel a_{q+i}$ for each $0\le i<t$.
\end{lemma}
\begin{proof}
	We have $b_{p+i}\not<_P a_{q+i}$ by choice of $b_{p+t}$. Were we to have $a_{q+i}<_P b_{p+i+1}$, we could shift $a_q,\ldots, a_{q+i}$ left one position in the $r$ alignment of $A$ and we would still have a weak $r$ alignment, contradicting Proposition \ref{prop:algn_char}.
	
	On the other hand, $a_{q+i}\not>_P b_{p+i+1}$ by property (3) of weak $r$ alignments.
\end{proof}

\begin{lemma}\label{lem: ecol}
	Applying $e_r$ to a $P$-array $A$ does not change the column of any element in the $r$ alignment.
\end{lemma}
\begin{proof}
	Let $a_1<_P a_2<_P\cdots<_Pa_m$ and $b_1<_Pb_2<_P\cdots<_P b_n$ be the entries of rows $r$ and $r+1$ of $A$ respectively. Proposition \ref{prop: raise_def} demonstrates that rows $r$ and $r+1$ in $e_r(A)$ are given by\[a_1,\ldots,a_{q-1}, b_p,\ldots, b_{p+t}, a_{q+t},\ldots, a_m\] and \[ b_1,\ldots, b_{p-1},a_q,\ldots, a_{q+t-1},b_{p+t+1},\ldots, b_n\] where $p$ and $t$ are as in the definition of $e_r$, and $q\ge 1$ is minimal such that $a_q$ either does not exist or lies in a column strictly right of $b_p$ in the $r$ alignment of $A$. Define
	\[
	\phi:\{ a_1,\ldots, a_m,b_1,\ldots,b_n \}\to \{r,r+1\}\times\Z_+
	\]
	so that it agrees with the $r$ alignment of $A$, except on $b_p,\ldots, b_{p+t}$ which are sent to row $r$, and $a_q,\ldots, a_{q+t-1}$ which are sent to row $r+1$. This amounts to swapping the rows in Fig. \ref{fig:raise}. We must show that $\phi$ is the $r$ alignment of $e_r(A)$.
		
	We see that $\phi$ inherits from the $r$ alignment of $A$ properties (1), (2), and (4) in Definition \ref{defn: weak_align}. Any entry in row $r$ of $\phi$ weakly left of $b_i$ for some $1\le i<p$ is also in row $r$ weakly left of $b_i$ in the $r$ alignment of $A$ and therefore cannot violate (3). The only elements in row $r$ of $\phi$ weakly left of $b_i$ for $p+t<i\le n$ are either in the same position in the alignment of $A$, or are some $b_j<_P b_i$. Either way, such $b_i$ cannot be part of a violation of (3). Each $a_{q+i}$ with $0\le i< t$ shares a column in $\phi$ with an incomparable element by Lemma \ref{lem:flip_raise} and so cannot violate (3) either. So $\phi$ is a weak $r$ alignment.
	
	Suppose we have a weak $r$ alignment $\psi$ such that $\psi\preceq \phi$. We will show that $\psi=\phi$. Since $b_p$ is the rightmost entry in row $r+1$ of the $r$ alignment of $A$ that does not share a column with an entry in row $r$, we know from Lemma \ref{lem:sw} that $b_1,\ldots, b_{p+t},a_{q+t},\ldots,a_m$ occupy adjacent columns in the alignment of $A$, hence in $\phi$. Properties (1) and (3) of weak $r$ alignments force each element in this chain to lie strictly right of the previous one in $\psi$, so the elements of this chain occupy the same columns in $\psi$ and $\phi$.
	
	Let $x$ be some element of the chain $a_q,\ldots, a_{q+t-1},b_{p+t+1},\ldots, b_n$, and assume we have determined all elements greater than $x$ in the chain to occupy the same position in both $\phi$ and $\psi$. In this case, we have in fact determined by the previous paragraph that all entries greater than $x$, in the chain or otherwise, occupy the same position in both $\phi$ and $\psi$. If $q=m+1$, i.e. $a_m$ lies in a column left of $b_p$ in the $r$ alignment of $A$, then this chain is vacuous as there cannot exist $b_{p+1}$ in a column strictly right of $a_m$ by choice of $b_p$. Assuming $q\ne m+1$ we then know $x$ must lie weakly left of $a_m$ in $\phi$, hence in $\psi$, by choice of $b_p$. Then the position of $x$ is uniquely determined by (3) and (4) in both $\phi$ and $\psi$ to be one column left of the leftmost greater element $y$. Since $y$ occupies the same position in $\phi$ and $\psi$ by induction, so to does $x$.
	
	It remains to show $\phi$ and $\psi$ agree on $a_1,\ldots, a_{q-1}$. Define \[\xi:\{ a_1,\ldots, a_m,b_1,\ldots,b_n \}\to\{r,r+1\}\times\Z_{+}\] to agree with the $r$ alignment of $A$ (hence the column assignments of $\phi$), except on the positions of $a_1,\ldots, a_{q-1}$ where it instead agrees with $\psi$. We know $\xi$ places each $a_i$ left of $a_{i+1}$ so it satisfies (1). Column positions of $\psi$ agree with $\xi$ which therefore satisfies (2) and (4).
	
	The positions of each $a_i$ with $q\le i\le m$ and each $b_j$ agree between $\xi$ and the $r$ alignment of $A$, so no such $a_i$ can violate (3) in $\xi$. For $1\le i<q$, we can only have $b_j<_P a_i$ if $j<p$. The positions of these elements in $\xi$ agree with $\psi$ so cannot violate (3) either. Then $\xi$ is a weak $r$ alignment that precedes the $r$ alignment of $A$ according to $\preceq$. Then $\xi$ is, in fact, the $r$ alignment of $A$ by Proposition \ref{prop:algn_char}. In particular, for $i<q$ we have $\psi(a_i)=\xi(a_i)=\phi(a_i)$. Therefore $\phi=\psi$ so $\phi$ is the alignment of $e_r(A)$ by Proposition \ref{prop:algn_char}.
\end{proof}

\begin{proposition}\label{prop: component}
	Let $A$ be a $P$-array. Consider the induced subgraph $G$ of $\operatorname{inc}(P)$ restricted to vertices in rows $r$ and $r+1$ of $A$. Applying $f_r$ or $e_r$ to $A$ swaps the rows of the elements in a single connected component of $G$.
\end{proposition}
\begin{proof}
	We will prove this for $f_r$, with the argument for $e_r$ being symmetric. If $f_r$ swaps the row of some $x\in P$, and if $y\in P$ is incomparable to $x$, then $y$ cannot share a row with $x$ either before or after applying $f_r$ which is to say $f_r$ swaps the row of $y$ as well. Thus $f_r$ swaps the rows of elements in a union of connected components of $G$.
	
	Now we must show that all elements whose rows are swapped by $f_r$ belong to the same connected component. Let $a_1<_P\cdots<_P a_m$ and $b_1<_P\cdots<_P b_n$ be the entries of rows $r$ and $r+1$ of $A$ respectively. Let $a_{p},\ldots, a_{p+t}$ and $b_{q},\ldots, b_{q+t-1}$ be the entries whose rows will be changed by $f_r$, as in Fig. \ref{fig:lower}. We have that $a_{p+i+1}\parallel b_{q+i}$ for $0\le i<t$ by Lemma \ref{lem:flip_lower}. That is to say each column in Fig. \ref{fig:lower} is a subset of a connected component. This also implies $b_{q+i}\not <_P a_{p+i}$, and we have $b_{q+i}\not >_P a_{p+i}$ by definition of $f_r$. Thus there is an edge in $G$ between each column in Fig. \ref{fig:lower}. So all of $a_p,\ldots, a_{p+t}$ and $b_q,\ldots, b_{q+t-1}$ belong to the same connected component as desired.
\end{proof}

We have seen in Propositions \ref{prop: lower_def} and \ref{prop: raise_def} that $e_r$ and $f_r$ are well-defined and affect the weight map in the correct way. The following theorem then completes the verification that $(\parrays{P},e_r,f_r,\wt)$ is a crystal.
\begin{theorem}\label{thm: inverse}
	For $A,A'\in \parrays{P}$ we have $f_r(A)=A'$ if and only if $e_r(A')=A$.
\end{theorem}
\begin{proof}
	If $f_r(A)=A'$ then by Lemma \ref{lem: fcol} and Lemma \ref{lem:sw}, the leftmost entry $x$ in row $r$ of the $r$ alignment of $A$ that does not share a column with an entry in row $r+1$, is the rightmost element in row $r+1$ of the alignment of $A'$ that does not share a column with an entry in row $r$. By Proposition \ref{prop: component}, both $f_r$ and $e_r$ operate by swapping the rows of the entries in the connected component of $x$ of the induced subgraph of $\operatorname{inc}(P)$ restricted to vertices in rows $r$ and $r+1$ of $A$, or equivalently $A'$. Thus $e_r(A')=A$. The other direction is similar.
\end{proof}

Fig. \ref{fig:not_schur} shows an example of a poset $P$ and a subset of $\parrays{P}$ with edges given by the crystal.

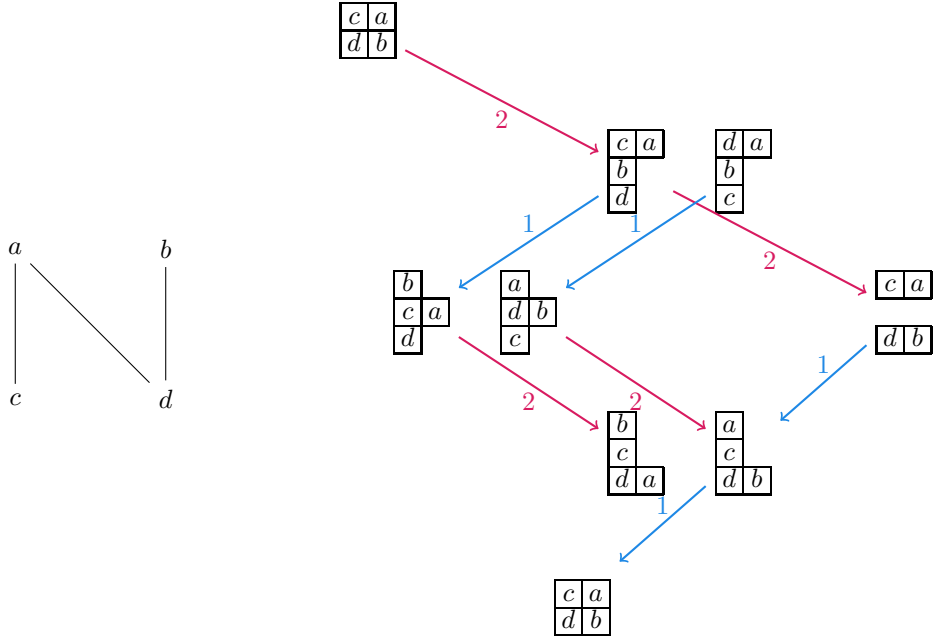
\begin{figure}[ht]
	\begin{tikzpicture}[scale=2,baseline=-150pt]
		\node (a) at (0,0) {$a$};
		\node (b) at (1,0) {$b$};
		\node (c) at (0,-1) {$c$};
		\node (d) at (1,-1) {$d$};
		\draw (a)--(c);
		\draw (a)--(d);
		\draw (b)--(d);
	\end{tikzpicture}
	\hspace{50pt}
	\begin{tikzpicture}[yscale=.75,xscale=.95]
		\node (a) at (-1.5,0) {\ytableausetup{centertableaux, boxsize=1em}\begin{ytableau} 
				c&a \\ 
				d&b\\
		\end{ytableau}};
		\node (b) at (2.25,-2.5) {\ytableausetup{centertableaux, boxsize=1em}\begin{ytableau} 
				c&a \\ 
				b\\
				d\\
		\end{ytableau}};
		\node (c) at (6,-5) {\ytableausetup{centertableaux, boxsize=1em}\begin{ytableau} 
				c&a \\ 
				\none\\
				d&b\\
		\end{ytableau}};
		\node (d) at (3.75,-7.5) {\ytableausetup{centertableaux, boxsize=1em}\begin{ytableau} 
				a \\ 
				c\\
				d&b\\
		\end{ytableau}};
		\node (e) at (1.5,-10) {\ytableausetup{centertableaux, boxsize=1em}\begin{ytableau} 
				\none \\ 
				c&a\\
				d&b\\
		\end{ytableau}};
		\node (f) at (-.75,-5) {\ytableausetup{centertableaux, boxsize=1em}\begin{ytableau} 
				b \\ 
				c&a\\
				d\\
		\end{ytableau}};
		\node (g) at (2.25,-7.5) {\ytableausetup{centertableaux, boxsize=1em}\begin{ytableau} 
				b \\ 
				c\\
				d&a\\
		\end{ytableau}};
		\node (h) at (.75,-5) {\ytableausetup{centertableaux, boxsize=1em}\begin{ytableau} 
				a \\ 
				d&b\\
				c\\
		\end{ytableau}};
		\node (i) at (3.75,-2.5) {\ytableausetup{centertableaux, boxsize=1em}\begin{ytableau} 
				d&a \\ 
				b\\
				c\\
		\end{ytableau}};
		\draw[thick,magenta2  ,->](a) -- (b)   node[midway,below]{$2$};
		\draw[thick,magenta2  ,->](b) -- (c)   node[midway,below]{$2$};
		\draw[thick,blue2  ,->](d) -- (e)   node[midway,above]{$1$};
		\draw[thick,magenta2  ,->](f) -- (g)   node[midway,below]{$2$};
		\draw[thick,magenta2  ,->](h) -- (d)   node[midway,below]{$2$};
		\draw[thick,blue2  ,->](b) -- (f)   node[midway,above]{$1$};
		\draw[thick,blue2  ,->](c) -- (d)   node[midway,above]{$1$};
		\draw[thick,blue2  ,->](i) -- (h)   node[midway,above]{$1$};
	\end{tikzpicture}
	\caption{\label{fig:not_schur} A segment of the $P$-array crystal for the given poset.}
\end{figure}

In the context of our crystal on $\parrays{P}$, the $P$-tableaux now hold crystal-theoretic significance.
\begin{theorem}\label{thm: hwt}
    The highest weight elements of $\parrays{P}$ are the $P$-tableaux.
\end{theorem}
\begin{proof}
    Let $T\in\parrays{P}$ and take $r\ge 1$ to be arbitrary. Let $a_1<_P a_2<_P\cdots<_Pa_m$ and $b_1<_Pb_2<_P\cdots<_P b_n$ be the entries of rows $r$ and $r+1$ of $T$ respectively. Suppose $e_r(T)=0$. Then every nonempty column in the $r$ alignment of $T$ contains some $a_i$. By property (2) of weak $r$ alignments this means each $a_i$ is in column $i$ of the $r$ alignment. Each $b_i$ is necessarily in column $i$ or greater, so $b_i\not<_P a_i$ by property (3). This shows that rows $r$ and $r+1$ of $T$ satisfy the required column condition of $P$-tableaux.
    
    Conversely, suppose $m\ge n$ and $b_i\not<_P a_i$ for each $1\le i\le m$. Then each $a_i$ gets mapped to column $i$ in the $r$ pre-alignment of $T$. As entries shift right to obtain the $r$-alignment, the chain $a_1,\ldots, a_m$ must stay put, and the entries $b_1,\ldots,b_n$ will never move past the rightmost occupied column $m$. Thus, every nonempty column in the $r$ alignment of $T$ contains some $a_i$ which means $e_r(T)=0$.
    
    We have shown that $e_r(T)=0$ is equivalent to the $P$-tableau column condition on elements in rows $r$ and $r+1$. Therefore $T$ is a highest weight element if and only if it is a $P$-tableau.
\end{proof}

\section{$s$-Positivity}\label{sec: spos}
Note that the crystal on $\parrays{P}$ induces a crystal on any subset. In this section we will be studying characters of connected components of $\parrays{P}$. If $C$ is a connected component, we will use the notation $C_\mathbf{a}$ to denote the set of $P$-arrays in $C$ with weight $\mathbf{a}$.

\begin{proposition}
	If $C$ is a connected component of $\parrays{P}$ then $\ch(C)$ is a symmetric function.
\end{proposition}
\begin{proof}
	We must show that the cardinalities of $C_\mathbf{a}$ and $C_\mathbf{b}$ match when $\mathbf{b}$ is a permutation of $\mathbf{a}$. It suffices to assume $\mathbf{b}=s_r\mathbf{a}$ for a simple transposition $s_r$, and $\mathbf{a}_r>\mathbf{a}_{r+1}$ without loss of generality.
	
	In the $r$ alignment of any $A\in C_\mathbf{a}$, we must have that there are at least $k=\mathbf{a}_r-\mathbf{a}_{r+1}$ entries in row $r$ that do not share columns with entries in row $r+1$. Therefore $f_r^k(A)\ne 0$ and \[\wt(f_r^k(A))=\mathbf{a}-k\alpha_r=s_r\mathbf{a}\] by Lemma \ref{lem: fcol}. Then $f_r^k$ provides a bijection from $C_{\mathbf a}$ to $C_{s_i\mathbf a}$ with inverse $e_r^k$ by Theorem \ref{thm: inverse}.
\end{proof}

The next result is our awaited refinement of Theorem \ref{thm:gash}. The proof follows in the spirit of Gasharov's by using the Jacobi-Trudi identity and the usual inner product on symmetric functions to get an alternating formula for the Schur coefficients in terms of the monomial coefficients \cite{Gas96}. However, the involution we use to cancel terms is based on our crystal operators, which allows us to stay within a connected component of the crystal.
\begin{theorem}\label{thm: spos}
	Let $C$ be a connected component of $\parrays{P}$. Then $\ch(C)=\sum_{\lambda} c_\lambda s_\lambda$ where $c_\lambda$ is the number of $P$-tableaux in $C$ of weight $\lambda$.
\end{theorem}
\begin{proof}
	Let $\sum_{\lambda} c_\lambda s_\lambda$ be the expansion of $\ch(C)$ into the Schur basis for symmetric functions. Set $n=\# P$ and let $\lambda=(\lambda_1,\ldots, \lambda_\ell)$ be a partition of $n$. For $\pi\in S_\ell$ we let $\pi(\lambda)$ denote the weak composition with $i$th component $\pi(\lambda)_i= \lambda_{\pi(i)}-\pi(i)+i$ for $1\le i\le \ell$. Notice that $\pi(\lambda)\ne\sigma(\lambda)$ when $\pi\ne \sigma$, as we have $\pi(i)<\sigma(i)$ for some $i$ which yields \[\pi(\lambda)_i-\sigma(\lambda)_i=(\lambda_{\pi(i)}-\lambda_{\sigma(i)})+(\sigma(i)-\pi(i))>0.\]
	The Jacobi-Trudi identity says
	\[
	s_\lambda = \det(h_{\lambda_j-j+i})_{i,j=1}^\ell = \sum_{\pi\in S_\ell}\operatorname{sgn}(\pi)h_{\pi(\lambda)}
	\]
	where $h_\mathbf{a}$ here denotes the complete homogeneous symmetric function of the weak composition $\mathbf{a}$. If $\mathbf{a}$ is not a weak composition, i.e. it has a negative part, then $h_{\mathbf a}$ is considered to be zero. Recall that the Schur functions are an orthonormal basis with respect to the inner product on symmetric functions defined by $\langle m_\lambda, h_\mu\rangle = \delta_{\lambda\mu}$ where $m_\lambda$ stands for the monomial symmetric function indexed by $\lambda$. Therefore
	\begin{equation}\label{eq: coeff}
	c_\lambda =\langle \ch(C), s_\lambda\rangle = \langle \ch(C), \sum_{\pi\in S_\ell}\operatorname{sgn}(\pi)h_{\pi(\lambda)}\rangle =\sum_{A\in\bigsqcup_{\pi\in S_\ell} C_{\pi(\lambda)}} \operatorname{sgn}(A)
	\end{equation}
	where $\operatorname{sgn}(A)$ takes the value $\operatorname{sgn}(\pi)$ when $A\in C_{\pi(\lambda)}$.
	
	We will define an involution $\iota: \bigsqcup_{\pi\in S_\ell} C_{\pi(\lambda)}\to \bigsqcup_{\pi\in S_\ell} C_{\pi(\lambda)}$.  Let $A=\{ A_{i,j} \}\in C_{\pi(\lambda)}$ for some $\pi\in S_\ell$. If $A$ is a $P$-tableau then set $\iota(A)=A$. Otherwise, there is a minimal $c$ such that some $A_{i+1,c}$ exists with $i>0$, but either $A_{i, c}$ does not exist or $A_{i,c}>_P A_{i+1,c}$. With $c$ so fixed, let $r$ be maximal among such $i$. 
	
	For convenience denote rows $r$ and $r+1$ of $A$ by $a_1<_P\cdots<_P a_{\pi(\lambda)_r}$ and $b_1<_P\cdots<_P b_{\pi(\lambda)_{r+1}}$.
	For $k<c$ we have $a_k\not>_p b_k$ which puts $a_k$ in column $k$ of the $r$ pre-alignment of $A$. If $b_{c+1}$ does not exist, then $a_c$ is placed in column $c_1$ of the $r$ pre-alignment. This is to say column $c+1$ either contains $a_c$ or $b_{c+1}$ or is emtpy. Since $P$ is (3+1)-free and $a_k\not>_P b_k$ when $k<c$, we have that such $a_k$ are less than $b_{c+1}$, if it exists. Then every entry in a column strictly left of $c+1$ in the $r$ pre-alignment is less than any entry in column $c+1$. This prohibits any of these lesser entries from occupying column $c+1$ in the $r$ alignment, as none of them can enter column $c+1$ until both $a_c$ and $b_{c+1}$ have vacated, but neither can an entry be moved into a previously empty column. This is to say that in the $r$ alignment, the entries weakly left of column $c$ are exactly $a_1,\ldots, a_{c-1}$ and $b_1,\ldots, b_c$.
	
	The remaining entries $a_c,\ldots, a_{\pi(\lambda)_r}$ and $b_{c+1},\ldots, b_{\pi(\lambda)_{r+1}}$ now lie strictly right of column $c$ in the $r$ alignment. So if $\pi(\lambda)_r\ge \pi(\lambda)_{r+1}$ then there are at least
	\[\pi(\lambda)_r-\pi(\lambda)_{r+1}+1=\pi(\lambda)_r-(\pi s_r)(\lambda)_r\] 
	such $a_i$ that do not share a column with an entry in row $r+1$. Then it makes sense to define \[\iota(A)=f_{r}^{\pi(\lambda)_r-(\pi s_r)(\lambda)_r}(A).\]
	If instead $\pi(\lambda)_r< \pi(\lambda)_{r+1}$ then there are at least 
	\[
	\pi(\lambda)_{r+1}-\pi(\lambda)_r-1= (\pi s_r)(\lambda)_r-\pi(\lambda)_r
	\] entries $b_i$ strictly right of column $c$ that do not share a column with an entry in row $r$. Here, it makes sense to define
	\[
	\iota(A)=e_{r}^{(\pi s_r)(\lambda)_r-\pi(\lambda)_r}(A).
	\]
	In either of these cases, note that $A$ gets mapped into $C_{(\pi s_r)(\lambda)}$. This means both that $\iota$ reverses the sign of $A$, and our codomain is correct.
	
	In the case where $\pi(\lambda)_r\ge\pi(\lambda)_{r+1}$ we have seen that the $\pi(\lambda)_r-(\pi s_r)(\lambda)_r$ entries in row $r$ of the $r$ alignment that do not share a column with an entry in row $r+1$ lie strictly right of column $c$. Similarly, in the case where $\pi(\lambda)_r< \pi(\lambda)_{r+1}$, we know the rightmost $(\pi s_r)(\lambda)_r-\pi(\lambda)_r$ in row $r+1$ of the $r$ alignment that do not share a column with an entry in row $r$ all lie strictly right of column $c$. In both cases, $a_1,\ldots,a_{c-1}$ and $b_1,\ldots, b_c$ lie weakly left of column $c$. So $\iota$ will not change the rows of any $a_1,\ldots, a_{c-1}$ or $b_1,\ldots, b_{c}$. That is, $\iota$ does not change the indices of any $A_{i,j}$ with $j<c$, or with $j=c$ and $i>r$. The selection of both $c$ and $r$ between $A$ and $\iota(A)$ therefore agree, and to apply $\iota^2$ to $A$ is to apply \[e_{r}^{(\pi s_r s_r)(\lambda)_r-(\pi s_r)(\lambda)_r}\circ f_{r}^{\pi(\lambda)_r-(\pi s_r)(\lambda)_r}\] or \[f_{r}^{(\pi s_r)(\lambda)_r-(\pi s_rs_r)(\lambda)_r}\circ e_{r}^{(\pi s_r)(\lambda)_r-\pi(\lambda)_r}\] which are both the identity map by Theorem \ref{thm: inverse}.
	
	Since $\iota$ reverses the sign of $P$-arrays that are not $P$-tableaux, terms cancel in equation \ref{eq: coeff} to get
	\[
	c_\lambda = \sum_{T}\operatorname{sgn}(T)
	\]
	summed over $P$-tableaux in $\bigsqcup_{\pi\in S_\ell} C_{\pi(\lambda)}$. If $\pi(i)>\pi(i+1)$ then $\pi(\lambda)_i<\pi(\lambda)_{i+1}$ which means each such $P$-tableau is in $C_\lambda$ as it must have partition weight. Furthermore, these $P$-tableaux have positive sign so we conclude $c_\lambda$ is the number of $P$-tableaux of weight $\lambda$ in $C$.
\end{proof}

Recall that by Proposition \ref{prop: component}, our crystal operators, and therefore our involution $\iota$, work by flipping connected components of the incomparability graph restricted to relevant rows. This is necessary if we want to obtain another $P$-array from our input. As a quick point of comparison between our involution and previous variations used to obtain coarser $s$-positivity theorems, Gasharov would have us flip all such connected components right of what we called column $c$ \cite{Gas96}, while Shareshian and Wachs would have us flip the subset of \textit{those} connected components which have odd cardinality \cite[Proof of Thm~6.3]{SW16}. Our involution is even more restrained, flipping a minimal number of these components.

\section{Natural Unit Interval Orders}\label{sec: nuio}
A \textbf{unit interval order} is a poset isomorphic to a finite subset $P$ of $\R$ with the relation $u<_P v$ when $u+1<v$. These are axiomatized by the requirements that they be (3+1)-free and (2+2)-free, which can be seen as a modification to the argument presented in \cite{SS58} or more directly in \cite{BB03}.

\begin{definition}
	A \textbf{natural unit interval order} is a finite poset $(P,\le_P)$ on a subset of $\N$ such that 
	\begin{itemize}
		\item $u<_P v$ implies $u<v$ as natural numbers, and
		\item if $u<_P w$ with $v\parallel w$ and $v\parallel u$, then $u<v<w$ as natural numbers.
	\end{itemize}
\end{definition}

These define the same isomorphism classes as the unit interval orders \cite{SW16}.

\begin{definition}\cite{SW16}
	Given a finite graph $G=(V,E)$ with $V\subset\N$, define the \textbf{chromatic quasisymmetric function} by
	\[
	X_G(x,q)=\sum_{\kappa} q^{\asc(\kappa)} \prod_{v\in V} x_{\kappa(v)}
	\]
	summed over all proper colorings of $G$, where
	\[
	\asc(\kappa)=\#\{ (u,v)\in E\mid \kappa(u)<\kappa(v)  \text{ and } u<v \text{ as natural numbers} \}.
	\]
	An edge that counts towards $\asc(\kappa)$ is an \textbf{ascent}.
\end{definition}
We of course recover the chromatic symmetric function as $X_G(x)=X_G(x,1)$.

If $P$ is a (3+1)-free poset on a subset of $\N$, then for $A\in\parrays{P}$ we write $\asc(A)$ to mean $\asc(\kappa)$ where $\kappa$ is the proper coloring of $\operatorname{inc}(P)$ that corresponds to $A$ by sending row $i$ of $A$ to the number $i$. In other words, $\asc(A)$ is the number of pairs $(x,y)\in P\times P$ such that $x\parallel y$, $x<y$ as natural numbers, and $x$ occurs in a higher row of $A$ than $y$. When $P$ is a natural unit interval order we will see that the crystal on $\parrays{P}$ respects this statistic.
\begin{theorem}
	If $P$ is a natural unit interval order then $\asc$ is constant on connected components of the $P$-array crystal.
\end{theorem}
\begin{proof}
	Let $A\in\parrays{P}$ with $f_r(A)\ne 0$. We must show $\asc(A)=\asc(f_r(A))$.
	
	Let $G$ be the graph induced by $\operatorname{inc}(P)$ by restricting to the vertices in rows $r$ and $r+1$ of $A$. By Proposition \ref{prop: component}, $f_r(A)$ is obtained from $A$ by swapping the rows of all elements in some connected components $C$ of $G$. Because $G$ is bipartite and $P$ is (3+1)-free, the degree of each vertex is at most 2 and therefore $C$ is a path or a cycle. Since $\wt(f_r(A))=\wt(A)-\alpha_r$ we know $C$ contains some $t+1$ elements in row $r$ of $A$, and $t$ elements in row $r+1$. Since $G$ is bipartite and $C$ has an odd number of elements, $C$ is in fact a path.
	
	Suppose we have some $u\in P\setminus C$ and $v\in P$ with $u\parallel v$. If $v\notin C$ then neither element changes rows between $A$ and $f_r(A)$ so the pair $(u,v)$ is an ascent in both or neither of $A$ and $f_r(A)$. If $v\in C$ then $u$ cannot reside in either row $r$ or $r+1$ of $A$. Again, the pair $(u,v)$ is an ascent in both or neither of $A$ and $f_r(A)$.
	
	To show $\asc(A)=\asc(f_r(A))$ we must finally show that there is a 1-1 correspondence between ascents in $A$ and ascents in $f_r(A)$ which occur entirely within $C$. In both cases, we will there is exactly one ascent involving each given $v\in C$ from row $r+1$ of $A$ (equivalently row $r$ of $f_r(A)$), which is enough since each ascent must include exactly one such element. Since $C$ is a path containing $t+1$ elements from row $r$ and $t$ elements from row $r+1$ of $A$, each element in row $r+1$ has degree 2 in $G$. Then there are $u,w\in C$ in row $r$ of $A$ with $u<_P w$, and with $v$ incomparable to both. By definition of a natural unit interval order, $u<v<w$ as natural numbers. Then $(u,v)$ is an ascent in $A$ but not $f_r(A)$, and $(v,w)$ is an ascent in $f_r(A)$ but not $A$. These are the only elements in $C$ incomparable to $v$, so these are the only possible ascents contained in $C$ involving $v$.
\end{proof}
Now if $C$ is a connected component of the $P$-array crystal for a natural unit interval order, it makes sense to define $\asc(C)$ as $\asc(A)$ for any $A\in C$. Therefore we have the following corollary.
\begin{corollary}\label{cor: ch_exp}
	If $P$ is a natural unit interval order then \[X_{\operatorname{inc}(P)}(x,q)=\sum_{C}q^{\asc(C)}\ch(C)\] summed over connected components of the $P$-array crystal.
\end{corollary}

Still for a natural unit interval order, Shareshian and Wachs showed $X_{\operatorname{inc}(P)}(x,q)$ is symmetric in $x$, and in fact $s$-positive in the sense that the coefficient of each $q^i$ is $s$-positive \cite{SW16}. Corollary \ref{cor: ch_exp} together with Theorem \ref{thm: spos} provide a new proof of these facts.

\section{Two-Row $P$-Tableaux}\label{sec: two-row}
For the remainder of the paper we again assume only that the poset $(P,\le_P)$ is (3+1)-free. We have seen that we can refine $\parrays{P}$ into connected components which have $s$-positive characters, but these characters are still not single Schur functions in general. Take, for instance, the component shown in Fig. \ref{fig:not_schur} which contains two $P$-tableaux. 

However, we will develop an explicit method to write the Schur functions corresponding to certain $P$-tableaux as generating functions of disjoint connected subsets of $\parrays{P}$. These $P$-tableaux $T$, are those such that $\wt(T)_i=0$ for all $i>2$, and we call them \textbf{two-row $P$-tableaux}. If $T$ is a two-row $P$-tableau we also say the \textbf{diagram of $T$} is the image of the 1 alignment of $T$, which we denote $\D(T)$.

\begin{lemma}\label{lem: top_just}
    If $T$ is a two-row $P$-tableau then each cell in row 2 of $\D(T)$ has a cell directly above it in row 1.
\end{lemma}
\begin{proof}
    By Theorem \ref{thm: hwt} we have $e_1(T)=0$ from which the result follows.
\end{proof}

For $T$ a two-row $P$-tableau we have $s_{\wt(T)}=s_{\wt(\D(T))}$ which, by Lemma \ref{lem: top_just} and Theorem \ref{thm: diagrams}, is the character of the connected component of $\D(T)$ in the diagram crystal. In order to write $s_{\wt(T)}$ as the generating function for a set of $P$-arrays, we therefore want to associate the diagrams with highest weight $\D(T)$ to some $P$-arrays in a weight preserving way. Let $\mathcal D(T)$ denote this set of diagrams.

\begin{definition}\label{def: filling}
    Let $T$ be a two-row $P$-tableau, and $D\in\mathcal{D}(T)$. The \textbf{filling} or \textbf{diagram filling} of $D$ with respect to $T$ is a bijection $\plabel{D}{T}:P\to D$ constructed column by column from left to right as follows.
    
    Suppose we have determined the entries mapped to all cells strictly left of some column $c$. If column $c$ of the 1 alignment of $T$ contains a unique entry, we map it to the unique cell in column $c$ of $D$. Otherwise, column $c$ of the 1 alignment contains two entries $x_1,x_2\in P$ in rows 1 and 2 respectively. We assign $x_1$ and $x_2$ to the cells in column $c$ according to the first of the following rules whose prerequisites are met.
    \begin{enumerate}
        \item Suppose some $y\in P$ is mapped to the topmost cell $(r,c-1)$ in column $c-1$ of $D$, and that there is exactly one $x_i$ greater than $y$. Then we let $s\le r$ be maximal such that $(s,c)\in D$ and we map $x_i$ to $(s,c)$ and the remaining poset element to the remaining cell.
        
        \item Suppose some $y\in P$ is mapped to the lowest cell $(r,c-1)$ in column $c-1$ of $D$, and that there is exactly one $x_i$ greater than $y$. Then we let $s\le r$ be maximal such that $(s,c)\in D$ and we map $x_i$ to $(s,c)$ and the remaining poset element to the remaining cell.
        
        \item Map $x_1$ and $x_2$ to the upper and lower cells in column $c$ of $D$ respectively.
    \end{enumerate}
\end{definition}
\begin{proposition}
	Let $T$ be a two-row $P$-tableau, and $D\in\mathcal D(T)$. The diagram filling $\plabel{D}{T}$ is well-defined.
\end{proposition}
\begin{proof}
	If column $c>1$ contains two cells in $D$, then Proposition \ref{prop: col_pairs} implies each cell in column $c-1$ must be column $(c-1)$-paired. Then for any $(r,c-1)\in D$ there exists $s\le r$ with $(s,c)\in D$ as presumed in rules (1) and (2).
\end{proof}

We denote the set of diagram fillings by
\[
\df{P}=\{ \plabel{D}{T} \mid \text{ $T$ is a two-row $P$-tableau and } D\in\mathcal D(T) \}
\]
and the set of diagram fillings for a fixed two-row $P$-tableau $T$ by
\[
\df{P,T}=\{ \plabel{D}{T} \mid D\in\mathcal D(T) \}.\]

The point of this construction is that each row of a diagram filling increases in $\le_P$ from left to right, so the filling specifies a $P$-array. We will prove this shortly.

\begin{lemma}\label{lem: no_flip}
    Let $T$ be a two-row $P$-tableau, and $D\in\mathcal D(T)$. Let $x_1,x_2\in P$ be in the same column $c$ in the $1$ alignment of $T$ with $x_1$ in row 1 and $x_2$ in row 2. Suppose they satisfy the condition that whenever we have a distinct element $y<_Px_1$ we have $y <_P x_2$ as well. Then $\plabel{D}{T}(x_1)$ is in a row above $\plabel{D}{T}(x_2)$. In particular, this applies when $x_1<_P x_2$.
\end{lemma}
\begin{proof}
    The proof is by induction on the column index. If $c=1$ then the column is governed by rule (3) so we are done.
    
    Suppose $c>1$ and that the column's assignments are defined according to rule (2). This is to say there is some $y\in P$ assigned to the lowest cell $(r,c-1)$ in column $c-1$ of $D$, and $y<_P x_2$ since we cannot have only that $y<_P x_1$. We must have that $(r,c-1)$ is distinct from the topmost cell in column $c-1$ since rule (1) does not apply. By Proposition \ref{prop: col_pairs}, both cells in column $c-1$ are column $(c-1)$-paired which means that both cells in column $c$ are in rows weakly above $r$. Therefore, when we take $s\le r$ to be maximal such that $(s,c)\in D$ it must be the case that $(s,c)$ is the lowest cell in column $c$ and is the image of $x_2$.
    
    Finally suppose $c>1$ and that the column's assignments are defined according to rule (1). There is some $y_2\in P$ assigned to the topmost cell $(r,c-1)$ in column $c-1$, and we know that $y_2\not<_P x_1$. We therefore must have that $y_2$ is in row 2 of the 1 alignment of $T$. Let $y_1$ be the element in row 1 column $c-1$ of the 1 alignment. If we have some $z<_P y_1<_P x_1$ then also $z<_P y_2$ by the (3+1)-free condition, but the inductive hypothesis contradicts $\plabel{D}{T}(y_2)$ being the topmost cell in column $c-1$. So rule (1) never actually applies in these circumstances.
\end{proof}
\begin{remark}
    If $x_1$ and $x_2$ are comparable, then $x_1<_P x_2$ by property (3) of weak $r$ alignments, so Lemma \ref{lem: no_flip} applies.
\end{remark}

\begin{proposition}\label{prop: inc_rows}
	Let $T$ be a two-row $P$-tableau and $D\in\mathcal D(T)$. Suppose we have cells $(r_2,c_1),(r_1,c_2)\in D$ such that $r_1\le r_2$ and $c_1<c_2$. Further assume there is no distinct $(r,c)\in D$ with $r_1\le r\le r_2$ and $c_1\le c\le c_2$. Then $\plabel{D}{T}(x)=(r_2,c_1)$ and $\plabel{D}{T}(y)=(r_1,c_2)$ for some $x<_P y$.
\end{proposition}
\begin{proof}
	Let $x,y\in P$ be the entries mapped to the cells $(r_2,c_1)$ and $(r_1,c_2)$ respectively. If there is some column $d$ with $c_1<d\le c_2$ that contains only one entry in the 1 alignment, then $x<_P y$ by Lemma \ref{lem:col_ordering}.
	
	We may proceed under the assumption that there is no such column $d$, and we first deal with the subcase where $c_2=c_1+1$. Suppose $(r_2,c_1)$ is the topmost cell in column $c_1$. If every entry in column $c_2$ of the 1 alignment of $T$ is greater than $x$, we are done. Otherwise, by definition of weak $r$ alignments there is one entry in column $c_2$ that is greater than $x$, and one entry in column $c_2$ that is not. Rule (1) of the definition of diagram fillings therefore applies, and ensures that the entry greater than $x$ is assigned to $(r_1,c_2)$. That is, $x<_P y$.
	
	Suppose instead that there is $(r,c_1)\in D$ with $r<r_2$, hence $r< r_1$ by assumption. Say $\plabel{D}{T}(z)=(r,c_1)$. If column $c_2$ of the filling is defined according to rule (1), then there is a single entry $w$ in column $c_2$ of the 1 alignment of $T$ that is greater than $z$. We must have that $z$ and $w$ share a row in $T$. Moreover, $w\ne y$ since $w$ must be mapped to a row weakly less than $r<r_1$. So $x$ and $y$ share a row in $T$ implying $x<_P y$. If column $c_2$ of the filling is not defined according to rule (1), and every entry in column $c_2$ of the 1 alignment of $T$ is greater than $x$, we again have $x<_P y$ in particular. If column $c_2$ of the filling is not defined according to rule (1) and there is exactly one entry in column $c_2$ of the alignment greater than $x$, then column $c_2$ of the filling is defined according to rule (2). This ensures that the entry greater than $x$ is mapped to $(r_1,c_2)$, which is to say $x<_P y$.
	
	Now suppose $c_2> c_1+1$, still assuming there is no column $d$ with $c_1<d\le c_2$ containing just one cell. Proposition \ref{prop: col_pairs} implies that each cell in a column $d$ with $c_1\le d<c_2$ is column $d$-paired. This further implies that the topmost cell in each such column $d$ is in a row weakly below the topmost cell in column $d+1$. Then $(r_1,c_2)$ cannot be the topmost cell in column $c_2$, else there would be some $(r,c_2-1)\in D$ with $r_1\le r\le r_2$. So $(r_1,c_2)$ is the lowest cell in its column. Since the 1 alignment of $T$ is a weak 1 alignment, there is a chain $x<_P z<_P w$ of entries in columns $c_1,c_1+1,c_2$. If $y=w$ we are done, so assume not. We must have $y\not<_P w$ by Lemma \ref{lem: no_flip} since $(r_1,c_2)$ is the lowest cell in column $c_2$. Then $x<_P y$ by the (3+1)-free condition.
\end{proof}

In particular, if we apply Proposition \ref{prop: inc_rows} when $r_1=r_2$, we see that the entries in each row of the diagram fillings form a chain. We have therefore justified our earlier claim that each diagram filling gives us a $P$-array. Explicitly, for $L\in\df{P}$ we write $\arrayify(L)$ to mean the $P$-array $\{A_{i,j}\}$ where $L(A_{i,j})$ is the $j$th cell from the left in row $i$ of $\operatorname{Im}(L)$.
Let
\[\dfa{P}=\{ \arrayify(L) \mid L\in \df{P}\}\]
denote the set of $P$-arrays obtained from diagram fillings, and let
\[\dfa{P,T}=\{ \arrayify(L) \mid L\in \df{P,T}\}\] denote the set of $P$-arrays obtained from diagram fillings with a fixed two-row $P$-tableau $T$. 

\begin{lemma}\label{lem: highest_filling}
    Let $T$ a two-row $P$-tableau. We have $\arrayify(\plabel{\D(T)}{T})=T$.
\end{lemma}
\begin{proof}
    If some column $c$ of $\plabel{\D(T)}{T}$ is not defined according to rule (1) or (2) in Definition \ref{def: filling}, then the column assignments of entries in this column agree with $T$. Rules (1) and (2) only apply if there are some $x,y$ in columns $c-1$ and $c$ of the 1 alignment of $T$ respectively with $x\not<_P y$. Then $y$ must be placed in the row opposite to $x$ in both $T$ and $\arrayify(\plabel{\D(T)}{T})$. By induction, we see that $T$ and $\arrayify(\plabel{\D(T)}{T})$ must agree on the rows of all elements in column $c$ of the 1 alignment of $T$.
\end{proof}

\begin{lemma}\label{lem: ascend}
    Let $T$ be a two-row $P$-tableau and $D\in\mathcal D(T)$. Let $x\in P$ and say $\plabel{D}{T}(x)=(r,c)$. If column $c+1$ of $D$ contains two cells, then there exists $y>_P x$ with $\plabel{D}{T}(y)=(s,c+1)$ for some $s\le r$.
\end{lemma}
\begin{proof}
    Suppose $(r,c)$ is the topmost cell in column $c$ of $D$. If there is exactly one $y>_P x$ in column $c+1$ of $\plabel{D}{T}$ then rule (1) in the definition of $\plabel{D}{T}$ ensures it resides in a row weakly above $r$. Otherwise, both elements in column $c+1$ are greater than $\plabel{D}{T}(r,c)$, with at least one still being weakly above row $r$ by Proposition \ref{prop: col_pairs}.
    
    Suppose instead that $(r,c)$ is the lower cell in column $c$ and that it is distinct from the topmost cell. Then all cells in column $c+1$ of $D$ are weakly above row $r$ by Proposition \ref{prop: col_pairs}, and at least one such cell corresponds to an entry greater than $x$ by Proposition \ref{prop:algn_char}.
\end{proof}
\begin{lemma}\label{lem: nw}
    Let $T$ be a two-row $P$-tableau and $D\in\mathcal D(T)$. Suppose for some $x,y\in P$ we have $\plabel{D}{T}(x)=(r_1,c_1)$ and $\plabel{D}{T}(y)=(r_2,c_2)$ with $r_1\le r_2$ and $c_1\le c_2$. Then $x\not>_P y$.
\end{lemma}
\begin{proof}
    If there is some column $d$ with $c_1<d\le c_2$ that contains only one cell in $D$, then the result follows by Lemma \ref{lem:col_ordering}. So assume not.
    
   	In light of Lemma \ref{lem: ascend} we may take a chain $x=x_1<_Px_2<_P\cdots<_Px_k$ where each consecutive element is one column to the right, and in a row weakly above the position of the previous element, and $x_k$ is in column $c_2$ of $\plabel{D}{T}$. If $x_k= y$ we are done, and if not then $x_k$ is in a weakly row above $r_1$ and therefore strictly above $r_2$. Now $y<_P x<_P x_k$ would be a contradiction of Lemma \ref{lem: no_flip}.
\end{proof}

\begin{lemma}\label{lem: dfa_align}
	Let $T$ be a two-row $P$-tableau and let $D\in\mathcal D(T)\setminus \{\D(T)\}$. Let $r\ge 1$ be minimal such that $e_r(D)\ne 0$. Let $x\in P$ be the element such that $\plabel{D}{T}(x)$ is the rightmost cell in row $r+1$ that is not $r$-paired. Then $x$ is the rightmost entry in row $r+1$ of the $r$ alignment of $\arrayify(\plabel{D}{T})$ that does not share a column with an entry in row $r$.
\end{lemma}
\begin{proof}
	Let $a_1<_P\cdots <_P a_m$ and $b_1<_P\cdots <_P b_n$ be the entries of rows $r$ and $r+1$ in $\arrayify(\plabel{D}{T})$ respectively. Let $p$ be maximal such that $\plabel{D}{T}(b_p)$ is not $r$-paired. Say $\plabel{D}{T}(b_p)=(r+1,c)$.
	
	Suppose we have some $a_k$ with $k$ maximal such that $\plabel{D}{T}(a_k)$ is in a column strictly left of column $c$. Say $\plabel{D}{T}(a_k)=(r,c')$ and assume $a_k\not<_P b_p$. By Lemma \ref{lem: max_descent}, there can be no sequence of cells $(r_1,c_1),(r_2,c_2),(r_3,c_3)\in D$ with each $r_i<r_{i+1}$ and $c_i\le c_{i+1}$. Thus, $\plabel{D}{T}(a_k)$ cannot be $(r-1)$-paired without forming such a sequence together with the $(r-1)$-paired cell and $(r+1,c)$. By choice of $r$, we must then have $r=1$. Since $a_k\not<_P b_p$ we can say that column $c'+1$ of the 1 alignment of $T$ contains two entries using Lemma \ref{lem:col_ordering}. Then in $D$, we must have $(1,c'+1)$ column $c'$-paired with $\plabel{D}{T}(a_k)=(1,c')$ contradicting either the choice of $a_k$, or the fact that $\plabel{D}{T}(b_p)=(2,c)$ is not 1-paired. Then we must have $a_k<_P b_p$.
	
	Since $\plabel{D}{T}(b_p)$ is not 1-paired, each $\plabel{D}{T}(a_i)$ with $i\le k$ must be 1-paired to some $\plabel{D}{T}(b_j)$ with $j<p$. Any $\plabel{D}{T}(a_{k-i})$ must then be in a column weakly left of $\plabel{D}{T}(b_{p-i-1})$ for $i\ge 0$. Then $a_{k-i}\not>_P b_{p-i-1}$ by Lemma \ref{lem: nw}. In particular, this puts $a_k$ in a column strictly left of $b_p$ in the 1 pre-alignment of $\arrayify(\plabel{D}{T})$. In the 1 alignment of $\arrayify(\plabel{D}{T})$, $a_k$ remains in a column strictly left of $b_p$ since $a_k<_P b_p$.
	
	We must determine the positions of the remaining $a_i$ in the $r$ alignment of $\arrayify(\plabel{D}{T})$ (no longer assuming $r=1$ since $a_k$ as defined above may not exist). Let $q\ge 1$ be minimal such that either $a_q$ does not exist, or $\plabel{D}{T}(a_q)$ is in a column strictly right of $\plabel{D}{T}(b_p)$. Note that $n-p\le m-q+1$ since any $\plabel{D}{T}(b_i)$ with $i>p$ is $r$-paired with some $\plabel{D}{T}(a_j)$ with $j\ge q$ by choice of $p$. This also means each $\plabel{D}{T}(b_{p+i})$ with $i\ge 1$ is in a column weakly right of $\plabel{D}{T}(a_{q+i-1})$.
	
	Suppose $b_{p+k}\not<_P a_{q+k}$ for some $k\ge 0$, and $\plabel{D}{T}(b_{p+k})$ is in a column strictly left of $\plabel{D}{T}(a_{q+k})$. If $k>0$, further assume $\plabel{D}{T}(b_{p+k})$ shares a column with $\plabel{D}{T}(a_{q+k-1})$. We claim $b_{p+k+1}$ exists and $\plabel{D}{T}(b_{p+k+1})$ shares a column with $\plabel{D}{T}(a_{q+k})$. If $k=0$ so that $\plabel{D}{T}(b_p)$ does not share a column with any cell in row $r$, this fact is necessitated by Proposition \ref{prop: inc_rows} and our remark that $\plabel{D}{T}(b_{p+k+1})$ is in a column weakly right of $\plabel{D}{T}(a_{q+k})$.
	
	When $k>0$ we have $\plabel{D}{T}(a_{q+k-1})$ and $\plabel{D}{T}(b_{p+k})$ sharing a column. The former cannot be $(r-1)$-paired without violating Lemma \ref{lem: max_descent}, so we must have $r=1$ by choice of $r$. Since $b_{p+k}\not<_P a_{q+k}$, Lemma \ref{lem:col_ordering} implies $\plabel{D}{T}(a_{q+k})$ shares a column with another cell. This cell can only be in row 2, and therefore corresponds to $b_{p+k+1}$, again by Lemma \ref{lem: max_descent}.
	
	Therefore, there is some $\ell\ge 0$ such that (perhaps vacuously)
	\begin{enumerate}[(a)]
		\item $b_{p+i}\not<_P a_{q+i}$ for all $0\le i<\ell$,
		\item $\plabel{D}{T}(b_{p+i})$ and $\plabel{D}{T}(a_{q+i-1})$ share a column for all $1\le i<\ell$, and
		\item if $a_{q+\ell}$ exists then $b_{p+\ell}<_P a_{q+\ell}$.
	\end{enumerate}
	Additionally, if $\ell>0$ we know $r=1$.
	
	Recall that each $\plabel{D}{T}(b_{p+i})$ with $i\ge 1$ is weakly right of the column containing $\plabel{D}{T}(a_{q+i-1})$, and therefore $b_{p+i}\not<_P a_{q+i-1}$ by Lemma \ref{lem: nw}. We can now say that in the $r$ pre-alignment of $\arrayify(\plabel{D}{T})$, the entries $b_1,\ldots,b_{p+\ell},a_{q+\ell},\ldots, a_m$ lie in consecutive columns. If $a_{q+\ell}$ actually exists, then $a_m$ lies in the rightmost occupied column $p-q+m+1\ge n$. If $a_{q+\ell}$ does not exist, then $b_n$ lies in the rightmost occupied column. In either case, the entries $b_1,\ldots,b_{p+\ell},a_{q+\ell},\ldots, a_m$ will remain in consecutive column in the $r$ alignment of $\arrayify(\plabel{D}{T})$. Moreover, the entries $b_{p+\ell+1},\ldots, b_n$ will be among the columns containing $a_{q+\ell},\ldots, a_m$. 
	
	To complete the proof, it remains to show that $a_q,\ldots, a_{q+\ell-1}$ occupy columns $p+1,\ldots, p+\ell$ in the $r$ alignment. This relies, mainly, on showing that $a_{q+i}\not<_P b_{q+i+1}$ for each $0\le i<\ell$. This is all vacuous when $\ell=0$ so assume $\ell>0$ hence $r=1$.
	
	Assume to the contrary that $a_{q+i}<_P b_{q+i+1}$ for some $0\le i<\ell$. Lemma \ref{lem: no_flip} asserts that $a_{q+i}$ is in row 1, and $b_{q+i+1}$ in row 2 of $T$. In fact, $a_q,\ldots, a_{q+i}$ must all reside in row 1 of $T$, and $b_{p},\ldots, b_{q+i+1}$ in row 2 since $b_{p+j}\not<_P a_{q+j}$ for each $0\le j\le i$. There must then be some $y\in P$ in row 1 of $T$ such that $\plabel{D}{T}(y)$ shares a column with $\plabel{D}{T}(b_p)=(2,c)$. If we have some $x<_P y<_P a_q$, then since $b_p\not<_P a_q$, the (3+1)-free condition gives $x<_P b_p$. By Lemma \ref{lem: no_flip}, we then have $\plabel{D}{T}(y)$ in a row above $\plabel{D}{T}(b_p)$, which is to say row 1. This contradicts the choice of $p$ such that $\plabel{D}{T}(b_p)$ is not $r$-paired.
	
	Now $a_{q+\ell-1}\not<_P b_{p+\ell}$ and we have seen $b_{p+\ell}$ lies in column $p+\ell$ of the $r$ alignment of $\arrayify(\plabel{D}{T})$. We also saw that column $p+\ell+1$ is either empty or contains $a_{q+\ell}$. Property (4) of weak $r$ alignments therefore requires $a_{q+\ell-1}$ to be in column $p+\ell$. If $0<i<\ell$ and $a_{q+i}$ lies in column $p+i+1$, then by Property (4) again we must have $a_{q+i-1}$ in column $p+i$, as $a_{q+i-1}$ is not smaller than $b_{p+i}$ in column $p+i$.
	
	To summarize, the entries $a_{q},\ldots, a_m$ lie in columns $p+1,\ldots, p+m-q+1$ of the $r$ alignment of $\arrayify(\plabel{D}{T})$, while $b_{p+1},\ldots, b_n$ lie in a subset of those same columns. Since each $a_i$ with $i<q$ lies in a column strictly left of $b_p$ in the $r$ alignment, $b_p$ is the rightmost entry in row $r+1$ of the $r$ alignment of $\arrayify(\plabel{D}{T})$ that does not share a column with an entry in row $r$.
\end{proof}

\begin{proposition}\label{prop: raise_order}
	Let $T$ be a two-row $P$-tableau and let $D\in\mathcal D(T)\setminus\{\D(T)\}$. Let $r\ge 1$ be minimal such that $e_r(D)\ne 0$. Then $\arrayify(\plabel{e_r(D)}{T})=e_r(\arrayify(\plabel{D}{T}))$.
\end{proposition}
\begin{proof}
	Let $a_1<_P\cdots <_P a_m$ and $b_1<_P\cdots <_P b_n$ be the entries of rows $r$ and $r+1$ in $\arrayify(\plabel{D}{T})$ respectively. Let $p$ be maximal such that the cell corresponding to $b_p$ in $\plabel{D}{T}$ is not $r$-paired. Let $G$ be the induced subgraph of $\inc(P)$ on the vertices in $a_1,\ldots,a_m$ and $b_1,\ldots, b_n$. By Lemma \ref{lem: dfa_align} and Proposition \ref{prop: component}, we obtain $e_r(\arrayify(\plabel{D}{T}))$ from $\arrayify(\plabel{D}{T})$ by taking the vertices in the connected component of $G$ containing $b_p$, and swapping them between rows $r$ and $r+1$. We must show the same description holds for $\arrayify(\plabel{e_r(D)}{T})$.
	
	Say $\plabel{D}{T}(b_p)=(r+1,c)$. Notice that $\plabel{D}{T}$ and $\plabel{e_r(D)}{T}$ agree in all columns strictly left of $c$, as the diagram fillings are defined column by column from left to right, and $D$ and $e_r(D)$ coincide in those columns.
	
	Next we consider how column $c$ differs in $\plabel{D}{T}$ and $\plabel{e_r(D)}{T}$. We claim $\plabel{e_r(D)}{T}(b_p)=(r,c)$. This is clear if column $c$ of the $1$ alignment of $T$ contains a single entry. Otherwise, column $c$ of both diagram fillings will be defined according to the same rule (1), (2), or (3) in the definition. We see that $(r+1,c)$ is the topmost cell in column $c$ of $D$ if and only if $(r,c)$ is the topmost cell in column $c$ of $e_r(D)$. So if column $c$ is defined by rule (3) we indeed get $\plabel{e_r(D)}{T}(b_p)=(r,c)$. 
	
	If $(r,c-1)\notin D$, then for any $(s,c-1)\in D$, $(r+1,c)$ is the lowest cell in column $c$ weakly above row $s$ in $D$ if and only if $(r,c)$ is the lowest cell in column $c$ weakly above row $s$ in $e_r(D)$. Similarly, if there is $(r',c)\in D$ distinct from $(r+1,c)$, then it is the lowest cell in column $c$ weakly above row $s$ in $D$ if and only if this is also true in $e_r(D)$. Therefore if column $c$ is defined according to rules (1) or (2), we will again get that $\plabel{e_r(D)}{T}(b_p)=(r,c)$.
	
	If instead $(r,c-1)\in D$, then also $(r+1,c-1)\in D$ as $(r+1,c)$ is not $r$-paired. In this case we must have $r=1$ by choice of $r$ because $(r,c-1)$ cannot be $(r-1)$-paired without violating Lemma \ref{lem: max_descent}. It would also violate Lemma \ref{lem: max_descent} to have a cell in row 3 or greater of column $c$ in $D$, so we are back in the case where column $c$ of the diagram fillings contains only the entry $b_p$.
	
	Let $q\ge 1$ be minimal such that either $a_q$ does not exist, or lies in a column strictly right of $b_p$ in the 1 alignment of $T$. Suppose for some $k\ge 0$ and all $0\le i\le k$ we have
	\begin{enumerate}[(a)]
		\item $\plabel{D}{T}(b_{p+i})=(r+1,c+i)$ and $\plabel{e_r(D)}{T}(b_{p+i})=(r,c+i)$, and
		\item if $i>0$ then $\plabel{D}{T}(a_{q+i-1})=(r,c+i)$ and $\plabel{e_r(D)}{T}(a_{q+i-1})=(r+1,c+i)$.
	\end{enumerate}
	We will show that if $a_{q+k}$ exists with $\plabel{D}{T}(a_{q+k})=(r,c+k+1)$, and $b_{p+k}\not<_P a_{q+k}$ then the same hypotheses apply for $k+1$, and otherwise that column $c+k+1$ is identical in $\plabel{D}{T}$ and $\plabel{e_r(D)}{T}$.
	
	First assume $\plabel{D}{T}(a_{q+k})=(r,c+k+1)$ and $b_{p+k}\not<_P a_{q+k}$. The latter assumption with Proposition \ref{prop: inc_rows} implies either $(r,c+k)\in D$ or $(r+1,c+k+1)\in D$. In either case, we have cells in rows $r$ and $r+1$ of the same column, so the cell in row $r$ cannot be $(r-1)$-paired without violating Lemma \ref{lem: max_descent}. This means $r=1$ by choice of $r$.
	
	If we are in the case $(1,c+k)=(r,c+k)\in D$ then $k>0$ hence $\plabel{D}{T}(a_{q+k-1})=(1,c+i)$. Since $b_{p+k}\not<_P a_{q+k}$, $\plabel{D}{T}(a_{q+k})=(1,c+k+1)$ cannot be the only cell in its column using Lemma \ref{lem:col_ordering}. A cell in row 3 or greater of column $c+i+1$ would again contradict Lemma \ref{lem: max_descent} together with $(1,c+k)$ and $(2,c+k)$. Therefore we have $(r+1,c+k+1)=(2,c+k+1)\in D$ anyway. It must be the case that $\plabel{D}{T}(b_{p+k+1})=(2,c+k+1)$.
	
	Now $\plabel{e_r(D)}{T}$ must bijectively assign $a_{q+k}$ and $b_{p+k+1}$ to the cells $(1,c+k+1)$ and $(2,c+k+1)$. We have $\plabel{e_r(D)}{T}(b_{p+k})=(1,c+k)$, and $b_{p+k}\not<_P a_{q+k}$. For $\plabel{e_r(D)}{T}$ to have increasing rows as per Proposition \ref{prop: inc_rows}, the only choice is $\plabel{e_r(D)}{T}(a_{q+k})=(2,c+k+1)$ and $\plabel{e_r(D)}{T}(b_{p+k+1})=(1,c+k+1)$. Then (a) and (b) hold true for $k+1$ as claimed.
	
	Assume $a_{q+k}$ does not exist or $\plabel{D}{T}(a_{q+k})$ is not in column $c+k+1$. Equivalently, $(r,c+k+1)\notin D$. We will show $\plabel{D}{T}$ and $\plabel{e_r(D)}{T}$ coincide in column $c+k+1$. If $b_{p+k+1}$ exists then by choice of $p$ the cells of $b_{p+1},\ldots, b_{p+k+1}$ are $r$-paired with a subset of the cells of $a_{q},\ldots, a_m$. In particular, $a_{q+k}$ would have to also exist and $\plabel{D}{T}(b_{p+k+1})$ lies in a column weakly right of $\plabel{D}{T}(a_{q+k})$, hence strictly right of column $c+k+1$. Therefore, we can see that $(r+1,c+k+1)\notin D$. If $k>0$ then we have seen $r=1$ and $(1,c+k),(2,c+k)\in D$. Using Lemma \ref{lem: max_descent}, there can be no cell in column $c+k+1$ of $D$ in row 3 or greater. Thus column $c+k+1$ of $D$ is vacant, and the two diagram fillings trivially agree in this column.
	
	Suppose instead $k=0$. Any $(s,c+1)\in D$ is the lowest cell in its column weakly above the row of $\plabel{D}{T}(b_p)=(r+1,c)$ in $D$ if and only if it is the lowest cell in its column weakly above the row of $\plabel{e_r(D)}{T}(b_p)=(r,c)$ in $e_r(D)$. Similarly if there is $(r',c)\in D$ distinct from $(r+1,c)$ then $(s,c+1)$ is the lowest cell in its column weakly above $r'$ in $D$ if and only if this is also the case in $e_r(D)$. It follows that $\plabel{D}{T}$ and $\plabel{e_r(D)}{T}$ will be defined the same way in column $c+1$.
	
	Finally, assume $\plabel{D}{T}(a_{q+k})=(r,c+k+1)$, but $b_{p+k}<_P a_{q+k}$. We will once again show the two diagram fillings agree in column $c+k+1$. If the column contains only one cell then this is immediate so we will assume it contains two. We have several cases.
	\begin{itemize}
		\item \textbf{Case: } $k=0$ and $(r+1,c+1)\notin D$\\
		Any $(s,c+1)\in D$ is the lowest cell in its column weakly above the row of $\plabel{D}{T}(b_p)=(r+1,c)$ in $D$ if and only if it is the lowest cell in its column weakly above the row of $\plabel{e_r(D)}{T}(b_p)=(r,c)$ in $e_r(D)$. If there is $(r',c)\in D$ distinct from $(r+1,c)$ then $(s,c+1)$ is similarly the lowest cell in its column below row $r'$ in $D$ if and only if the same is true in $e_r(D)$. Column $c+1$ will therefore be defined identically in $\plabel{D}{T}$ and $\plabel{e_r(D)}{T}$.
		
		\item \textbf{Case: } $k=0$ and $(r+1,c+1)\in D$\\
		If column $c+1$ is defined in $\plabel{D}{T}$ and $\plabel{e_r(D)}{T}$ by rule (3) we are done, so assume not.
		
		We must have $b_p$ less than the entry corresponding to $(r+1,c+1)$ in $D$, hence every entry in column $c+1$. If $\plabel{D}{T}(b_p)$ is the topmost cell in its column, then the conditions for defining column $c+1$ of $\plabel{D}{T}$ or $\plabel{e_r(D)}{T}$ according to rule (1) are not satisfied. If $\plabel{D}{T}(b_p)$ is the lowest cell in its column, then the conditions for defining column $c+1$ of $\plabel{D}{T}$ or $\plabel{e_r(D)}{T}$ according to rule (2) are not satisfied. Thus, there must be $(r',c)\in D$ distinct from $(r+1,c)$, and column $c+1$ of both diagram fillings is determined by mapping the unique entry in column $c+1$ of the 1 alignment of $T$ to the lowest cell in column $c+1$ weakly above row $r'$.
		
		\item \textbf{Case: } $k>0$\\
		We have seen that we must have $r=1$, and we know $(1,c+k),(2,c+k)\in D\cap e_r(D)$. Using Lemma \ref{lem: max_descent}, the two cells in column $c+k+1$ of $D\cap e_r(D)$ can only be $(1,c+k+1)$ and $(2,c+k+1)$. Thus $\plabel{D}{T}(b_{p+k+1})=(2,c+k+1)$. We know $b_{p+k-1}\not<_P a_{q+k}$ so the $(3+1)$-free condition implies $a_{q+k}<_P b_{p+k+1}$. Thus, every entry in column $c+k+1$ of the 1 alignment of $T$ is greater than every entry in column $c+k$. Then column $c+k+1$ of $\plabel{D}{T}$ and $\plabel{e_r(D)}{T}$ are defined according to rule (3).
	\end{itemize}
	
	Now we fix $k\ge 0$ to be maximal such that for all $0\le i<k$ hypotheses (a) and (b) hold. We have summarily shown that $\arrayify(\plabel{e_r(D)}{T})$ is obtained from $\arrayify(\plabel{D}{T})$ by putting $a_q,\ldots, a_{q+k-1}$ in row $r+1$ and putting $b_p,\ldots, b_{p+k}$ in row $r$. Together, these entries form a union $C$ of connected components of $G$, since $\arrayify(\plabel{e_r(D)}{T})$ is, in fact, a $P$-array. 
	
	To finish, we must show that $C$ is actually a single connected component of $G$, i.e. the connected component containing $b_p$. Assume $k>0$ hence $r=1$, else this is trivial. For every $0\le i<k$ we have $\plabel{D}{T}(a_{q+i})$ above $\plabel{D}{T}(b_{p+i+1})$ in the same column and $\plabel{e_r(D)}{T}(a_{q+i})$ below $\plabel{e_r(D)}{T}(b_{p+i+1})$ in the same column. Then we must have $a_{q+i}\parallel b_{p+i+1}$ by Lemma \ref{lem: no_flip}. We also know $b_{p+i}\not<_P a_{q+i}$ for each $0\le i<k$, and it now suffices to show $b_{p+i}\not>_P a_{q+i}$. If we had $b_{p+i}>_P a_{q+i}>_P a_{q+i-1}$ for $i>0$, this would contradict that $a_{q+i-1}\parallel b_{p+i}$ as just shown.
	
	Then assume $b_p>_P a_q$. If $b_p$ is the only entry in its column of the 1 alignment of $T$, then it is in row 1 of $T$ and $a_q$ in row 2. This is impossible by property (3) of weak $1$ alignments. Then there must be some $x\in P$ with $\plabel{D}{T}(x)=(s,c)$ and $s>2$. Because $b_p$ and $a_q$ cannot share a row in $T$, $x$ and $a_q$ do share a row. Thus $x<_P a_q<_P b_p$ but this contradicts Lemma \ref{lem: no_flip}. We conclude $C$ is a connected component of $G$ and we are done.
\end{proof}

\begin{theorem}\label{thm: schur}
	The map 
	\[
	\varphi: \bigsqcup_{T \text{ two-row $P$-tableau}} \{ (T,D) \mid D\in\mathcal{D}(T) \}\to \dfa{P}
	\]
	given by $(T,D)\mapsto \arrayify(\plabel{D}{T})$ is a bijection, and $\wt(D)=\wt(\varphi(T,D))$. Additionally, $\dfa{P,T}$ is a connected subset of $\parrays{P}$ for each two-row $P$-tableau $T$.
\end{theorem}
\begin{proof}
	We have that $\varphi$ is onto by definition of $\dfa{P}$, and it is also immediate that $\wt(D)=\wt(\varphi(T,D))$ whenever $D\in\mathcal{D}(T)$.
	
	Let $T$ be a two-row $P$-tableau and $D\in\mathcal{D}(T)$. If $D\ne\D(T)$ then we may take $r\ge 1$ to be minimal such that $e_r(D)\ne0$ and we get that $e_r(\varphi(T,D))=\varphi(T,e_r(D))$ by Proposition \ref{prop: raise_order}. Repeated application shows that some sequence of raising operations applied to $\varphi(T,D)$ gives us $\varphi(T,\D(T))=T$, by Lemma \ref{lem: highest_filling}, with the corresponding sequence of $P$-arrays being contained in $\dfa{P,T}$. Therefore $\dfa{P,T}$ is a connected subsets of the $P$-array crystal.
    
    Moreover, we claim that $r$ as chosen is minimal such that $e_r(\varphi(T,D))\ne 0$. To see this, let $1\le s<r$. Let $a_1<_P\cdots<_Pa_m$ and $b_1<_P\cdots<_Pb_n$ be the entries in rows $s$ and $s+1$ of $\varphi(T,D)$ respectively. By choice of $r$, the cell of each $b_i$ in $\plabel{D}{T}$ is $s$-paired. In particular, $a_i$ exists and must lie in a column weakly left of $b_i$. By Lemma \ref{lem: nw}, $a_i\not>_P b_i$ which confirms that $e_s(A)=0$.
    
    Therefore, the sequence of raising operations applied to $\varphi(T,D)$ to obtain $T$ depends only on $\varphi(T,D)$, not $D$ or $T$ per se. This is to say that if we have $\varphi(T,D)=\varphi(T',D')$ for some two-row $P$-tableau $T'$ and diagram $D'\in\mathcal{D}(T')$ then the same sequence of raising operations applied to $\varphi(T,D)$ yields $T'$, which thus coincides with $T$. 
    
    Finally, we must show that when $D$ and $D'$ are distinct diagrams in $\mathcal{D}(T)$ then $\varphi(T,D)\ne \varphi(T,D')$. We may take some $(r,c)\in D\setminus D'$ and we get that row $r$ of $\varphi(T,D)$ contains an entry from column $c$ of the 1 alignment of $T$, while row $r$ of $\varphi(T,D')$ contains no such entry. Therefore $\varphi$ is a bijection.
\end{proof}
\begin{corollary}
	For $T$ a two-row $P$-tableau we have \[s_{\wt(T)}=\sum_{A\in\dfa{P,T}} \prod_{i\ge 1} x_i^{\wt(A)_i}.\]
\end{corollary}
\begin{proof}
	Using Theorem \ref{thm: diagrams} and Theorem \ref{thm: schur} we have
	\[
	s_{\wt(T)}=s_{\wt(\D(T))}=\sum_{D\in\mathcal{D}(T)}\prod_{i\ge 1} x_i^{\wt(D)_i}=\sum_{A\in\dfa{P,T}} \prod_{i\ge 1} x_i^{\wt(A)_i}.
	\]
\end{proof}
Recall from Remark \ref{rem: young} that $\mathcal{D}(T)$ is in explicit weight preserving bijection with semi-standard Young tableau. Thus, Theorem \ref{thm: schur} can be thought of as a partial Robinson-Schensted correspondence for $P$-arrays. It would be nice to extend the result to a full bijection, though such an extension is unlikely to be straightforward. For one thing, our partial bijection has the notable property that its image for a fixed $P$-tableau is a connected subset of $\parrays{P}$, but after removing this subset for the two-row $P$-tableau in Fig. \ref{fig:not_schur} the remaining maximal connected subsets are not individually $s$-positive.

\bibliographystyle{amsalpha}
\bibliography{Crystal}
\end{document}